\documentclass[10pt]{article}

\pdfoutput=1

\usepackage{verbatim,amsopn,amsthm,amssymb,amsmath,graphicx,subfigure}

\theoremstyle{plain}
\newtheorem{thm}{Theorem}[section]
\newtheorem{lem}[thm]{Lemma}
\newtheorem{prop}[thm]{Proposition}
\newtheorem{cor}[thm]{Corollary}

\newtheorem{alg}[thm]{Algorithm}

\theoremstyle{definition}
\newtheorem{defn}[thm]{Definition}

\theoremstyle{remark}
\newtheorem*{remark}{Remark}
\newtheorem*{conv}{Convention}

\DeclareMathOperator{\hl}{hl}

\DeclareMathOperator{\lk}{lk}

\DeclareMathOperator{\AO}{\mathcal{AO}}
\DeclareMathOperator{\Aut}{Aut}
\DeclareMathOperator{\PQ}{\mathcal{PQ}}

\DeclareMathOperator{\cp}{cap}

\DeclareMathOperator{\aux}{aux}

\newcommand{\ds}{\displaystyle}

\newcommand{\fg}{\leq_{f\!g}}

\begin{document}

\title{Elliptic subgroups in cyclic splittings of a free group}
\author{Brent B. Solie}

\maketitle

\begin{abstract}
We use Gersten's generalization of Whitehead's algorithm to determine whether a given finitely generated subgroup of a free group $F$ is elliptic in an elementary cyclic splitting of $F$. We provide a similar result for all elementary cyclic splittings of a free group of rank two and partial results towards an algorithm for ellipticity in elementary cyclic splittings of $F$ as an HNN-extension.
\end{abstract}


\section{Introduction}

In 1936, J. H. C. Whitehead developed an algorithm to solve the following problem: Given elements $u$ and $v$ in a finitely generated free group $F(X)$, is there an automorphism $\phi \in \Aut(F(X))$ such that $\phi(u) = v$? In his solution to this problem, Whitehead observed that the cyclically reduced length of an element of $F(X)$ serves as a measure of the word's complexity and that a finite set of particularly simple generators of $\Aut(F(X))$, now known as Whitehead automorphisms, can be used to systematically calculate the shortest elements of an element's $\Aut(F(X))$ orbit. By comparing these `minimal' elements for both $u$ and $v$, one can deduce whether they share the same orbit.

Decades later, in 1980, Gersten extended Whitehead's algorithm from single elements of $F(X)$ to finitely generated subgroups. Gersten's generalization answers the analogous problem: Given finitely generated subgroups $H$ and $K$ of a finitely generated free group $F(X)$, is there an automorphism $\phi \in \Aut(F(X))$ such that $\phi (H)=K$? Gersten's principle insight is that the appropriate measure of complexity for a finitely generated subgroup of $F(X)$ is the number of edges in its Stallings graph. The set of Whitehead automorphisms again provide a systematic way of calculating the automorphic images of $H$ and $K$ with minimal complexity, answering the question of whether $H$ and $K$ share an orbit.

The present paper investigates how Gersten's extension of Whitehead's algorithm may be used to study the structure of elementary cyclic splittings of $F(X)$.
An elementary cyclic splitting of $F(X)$ is a decomposition of $F(X)$ as a graph of groups with exactly one edge whose edge group is infinite cyclic.
Thus elementary cyclic splittings come in one of two varieties: \emph{segment splittings}, which correspond to the decomposition of $F(X)$ as an amalgamated product $F(X) \cong H \ast_\mathbb{Z} K$, and \emph{loop splittings}, which correspond to the decomposition of $F(X)$ as an HNN-extension $F(X) \cong H \ast_\mathbb{Z}$.
Elements or subgroups are said to be \emph{elliptic} if they are conjugate to elements or subgroups of $H$ or $K$ in a segment splitting, or conjugate to elements or subgroups of $H$ in a loop splitting.

Our main result is the following:

\begin{thm}
  Let $F(X)$ be a free group of finite rank, and let $H$ be a finitely generated subgroup of $F(X)$.
  There is an algorithm to determine whether $H$ is elliptic in some segment splitting of $F(X)$ and, if it is, produce such a splitting.
\end{thm}

In the first section of the present paper, we review some preliminaries concerning free groups, Stallings graphs, splittings, and Whitehead's algorithm. In the second section, we present the main result, its proof, and some partial results in the case of loop splittings. 

\section{Background}

The exposition provided in this section on Stallings graphs and folding is standard. Complete details can be found, for instance, in \cite{Kapovich2002}.

Let $X$ be a finite set with at least two elements.
Define $X^{-1} := \{x^{-1} : x \in X\}$ to be the set of \emph{formal inverses} of elements of $X$, and set $X^\pm := X \sqcup X^{-1}$.
We denote the set of words on the letters $X^\pm$ by $(X^\pm)^*$.
A word in $(X^\pm)^*$ is \emph{freely reduced} if it has no subword of the form $xx^{-1}$ or $x^{-1}x$ for any $x \in X$.
A word in $(X^\pm)^*$ is \emph{cyclically reduced} if every cyclic permutation of that word is freely reduced.

Let $F(X)$ be the free group on the letters $X$.
The \emph{$X$-length} of $w \in F(X)$, denoted $|w|_X$, is the length of the freely reduced word in $(X^\pm)^*$ which represents $w$.
We will indicate that $H$ is a finitely generated subgroup of $F(X)$ by $H \fg F(X)$.

\subsection{Stallings Graphs}

\begin{defn}[$X$-digraph] \index{$X$-digraph}
  Given a finite set $X$, an \emph{$X$-digraph} is given by the data $(V, E, \cdot_+, \cdot_-, \lambda)$, where:
  \begin{itemize}
    \item $V$ and $E$ are sets;
    \item $\cdot_+, \cdot_- : E \rightarrow V$; and
    \item $\lambda: E \rightarrow X$.
  \end{itemize}
  We call $V$ the \emph{vertex set} and $E$ the \emph{edge set}.
  For $e \in E$, we say that $e_-$ is the \emph{initial vertex} of $e$ and $e_+$ is the \emph{terminal vertex} of $e$.
  We call \emph{lambda} the labeling function and $\lambda(e)$ the \emph{label} of edge $e$.

  Let $S$ be an $X$-digraph.
  By $VS$ and $ES$ we denote the vertex and edge sets of $S$, respectively.
  For $v \in VS$, we define the \emph{in-link} of $v$ to be $\lk_+(v) := \{e \in ES: e_+ = v\}$, and we say the \emph{in-hyperlink} of $v$ is $\hl_+(v) := \{\lambda(e) : e \in \lk_+(v) \}$.
  Likewise, we define the \emph{out-link} of $v$ to be $\lk_-(v) := \{e \in ES: e_- = v\}$ and the \emph{out-hyperlink} of $v$ to be $\hl_-(v) := \{ \lambda(e)^{-1} : e \in \lk_-(v)\}$.
  The \emph{link} of $v$ is $\lk(v) := \lk_-(v) \cup \lk_+(v)$ and the \emph{hyperlink} of $v$ is $\hl(v) := \hl_-(v) \cup \hl_+(v)$.

  We say that the \emph{degree} of $v \in VS$ is $\deg_S(v) := \# \lk(v)$.
  If $\deg_S(v) = 0$ we say that $v$ is \emph{isolated}, and if $\deg_S(v) = 1$ we say that $v$ is a \emph{leaf}.
  If no vertex of $S$ is a leaf, we say that $S$ is \emph{cyclically reduced}.

  If $u, v \in VS$ are such that there is $e \in ES$ with $e_- = v$ and $e_+ = u$, then we say that $u$ and $v$ are \emph{adjacent} and that $e$ is \emph{incident} to both $u$ and $v$.
  If $e, f \in ES$ are such that $e_- = f_-$, then we say that $e$ and $f$ are \emph{coinitial}; if $e_+ = f_+$, then $e$ and $f$ are \emph{coterminal}.
  The edges $e$ and $f$ are \emph{coincident} if $e$ shares an endpoint with $f$.

  Let $Y \subseteq X$.
  A \emph{$Y$-edge} is any edge with label in $Y$.
  The set of $Y$-edges of $S$ is denoted $E_Y S$.

  \begin{conv}
    To aid readability, we will always denote singleton sets by their unique element.
    For instance, if $x \in X$, we will write $x$-edges rather than $\{x\}$-edges, and the set of $x$-edges of $S$ will be denoted $E_x S$ rather than $E_{\{x\}} S$.
  \end{conv}

  We say that $S$ is \emph{folded at $v$} if $\lambda$ induces a bijection $\lk(v) \rightarrow \hl(v)$.
  The graph $S$ is \emph{folded} if $S$ is folded at every vertex.
  If $S$ is not folded, then some pair of coterminal or coinitial edges $e, f \in ES$ share the same label.
  We \emph{fold} these edges by identifying the pair of edge $e$ and $f$ and either the vertices $e_-$ and $f_-$ if $e$ and $f$ are coinitial or the vertices $e_+$ and $f_+$ if $e$ and $f$ are coterminal.
  The process of performing folds in $S$ until none remain is called \emph{folding}. [TODO: put in a reference on stallings graphs and assure the reader that the end result is an invariant reduced graph.]

  We may \emph{delete} a leaf $v$ of $S$ by deleting $v$ and the unique edge incident to it.
  By repeatedly deleting leaves, we eventually arrive at a cyclically reduced $X$-digraph.
  We call this process \emph{cyclic reduction}.

  Lastly, if $S$ and $T$ are $X$-digraphs, an \emph{$X$-map} is a map $S \rightarrow T$ which sends vertices to vertices, edges to edges, and preserves both orientation and label.
  We will assume that all of our maps of $X$-digraphs are $X$-maps.
  An $X$-map is an \emph{immersion} if the induced map on the link of a vertex is injective for every vertex of $S$.
\end{defn}

\begin{defn}[Dual digraph] \index{dual digraph}
  Given an $X$-digraph $S$, we may construct the \emph{dual} of $S$, denoted $S^*$, by adding a set of \emph{formal inverse edges} $\overline{E}S := \{\bar{e} : e \in ES\}$ and extending $\cdot_-, \cdot_+,$ and $\lambda$ as follows:
  \begin{align*}
    & (\bar{e})_- := e_+, \\
    & (\bar{e})_+ := e_-, \text{ and }\\
    & \lambda(\bar{e}) := \lambda(e)^{-1}.
  \end{align*}
  By defining $\overline{(\bar{e})} = e$, the above equations are satisfied for any $e \in ES^*$.

  A \emph{path} $p$ in an $X$-digraph $S$ is a sequence of edges $p = e_1 e_2 \dots e_l$ in the dual $S^*$ such that $(e_i)_+ = (e_{i+1})_-$ for all $i=1, \dots, l-1$.
  The path $p$ is a \emph{loop} if we further have that $(e_l)_+ = (e_1)_-$.
  The path is \emph{immersed} if $e_{i+1} \neq \bar{e}_i$ for all $i=1, \dots, l-1$.
  The \emph{label} of $p$ is $\lambda(p):= \lambda(e_1) \dots \lambda(e_l)$.
  The \emph{length} of $p$ is $l$.
  We say that an $X$-digraph $S$ is \emph{connected} if there exists a path between any two vertices.
\end{defn}

\begin{remark}
  For our purposes, we will not distinguish much between an $X$-digraph $S$ and its dual $S^*$.
  Specifically, for $x \in X$, we will regard an $x^{-1}$-edge as simply an $x$-edge with the opposite orientation.
  If $e$ is an $x$-edge, we will therefore consider it as an $x$-edge in that it contributes the label $x$ to the hyperlink of $e_+$, and also as an $x^{-1}$ edge as it contributes the label $x^{-1}$ to the hyperlink of $e_-$.
\end{remark}

\begin{defn}[Stallings graph] \index{Stallings graph}
  Let $H \fg F(X)$.
  The \emph{Stallings graph} representing $H$ with respect to $X$, denoted $S_X(H)$, is the unique $X$-digraph with basepoint such that a freely reduced word in $(X^\pm)^*$ represents an element of $H$ if and only if it occurs as the label of an immersed loop of $S_X(H)$ beginning and ending at the basepoint.

  Recall that we may construct $S_X(H)$ as follows.
  Let $h_1, \dots, h_k$ be elements of $(X^\pm)^*$ representing a finite set of generators of $H$.
  Beginning with a basepoint, denoted $1$, we construct a loop beginning and ending at $1$ with label $h_i$ for each $i$; let $S_0$ be the resulting graph.
  We then perform all possible folds in $S_0$ (in any order) to obtain a folded graph $S_1$.
  Finally, we repeatedly delete leaves of $S_1$ different from $1$ until no leaves remain except possibly the basepoint.
  The resulting graph is $S_X(H)$, and it is well-known that $S_X(H)$ is invariant with respect to the choice of generating set for $H$ as well as the order of the folds and leaf deletions.

  Suppose that $S_X(H)$ is cyclically reduced.
  For any $g \in F(X)$, we may construct $S_X(H^g)$ from $S_X(H)$ by adding a new basepoint $1'$, a path from $1$ to $1'$ labeled by $g$, and then folding and deleting non-basepoint leaves.
  Therefore whenever $S_X(H)$ is cyclically reduced, we will forget the basepoint and think of $S_X(H)$ as representing $H^{F(X)}$, the conjugacy class of $H$ in $F(X)$, rather than the single subgroup $H$.
\end{defn}

\subsection{Whitehead's Algorithm}

Our discussion of Gersten's extension of Whitehead's algorithm follows that of \cite{Roig2007}.

\begin{defn}[Whitehead automorphism] \index{Whitehead automorphism}
  \label{Defn--Whitehead automorphism}
  A \emph{type I Whitehead automorphism} \index{Whitehead automorphism} is an automorphism $\phi \in \Aut F(X)$ which is induced by permutations and inversions of the set $X^\pm$.

  A \emph{type II Whitehead automorphism} is an automorphism $\phi \in \Aut F(X)$ for which there exists $m \in X^\pm$ such that $\phi(m)=m$ and
  \begin{equation*}
    \phi(x) \in \{x, m^{-1}x, xm, x^m\}
  \end{equation*}
  for all $x \in X$.
  We call $m$ the \emph{multiplier} for $\phi$.

  Given a type II Whitehead automorphism $\phi$ with multiplier $m$, define
  \begin{equation*}
    C := \left\{ x \in X^\pm : \phi(x) \in \{m, xm, x^m\} \right\}.
  \end{equation*}
  Then $\phi$ is determined completely by the pair $(C,m)$, and we refer to $C$ as the \emph{cut} for $\phi$.

  Let $C \subseteq X^\pm$ be such that $m \in C$ and $m^{-1} \notin C$.
  We call such a $C$ an \emph{$m$-cut}.
  For any $m \in X^\pm$ and $m$-cut $C$, the pair $(C,m)$ defines a type II Whitehead automorphism of $F(X)$.

  More generally, if $C, D \subseteq X^\pm$, we say that $C$ \emph{cuts} $D$ if $D$ contains an element of both $C$ and $C' := X^\pm - C$.
\end{defn}

\begin{defn}[Hypergraph] \index{hypergraph}
  A \emph{hypergraph} is a tuple $(V,E,\iota)$, where $V$ and $E$ are sets and $\iota: E \rightarrow \mathcal{P}(V)$, where $\mathcal{P}(V)$ denotes the power set of $V$.
  The elements of $V$ are called \emph{vertices} and the elements of $E$ are called \emph{hyperedges}.
  We call $\iota$ the \emph{incidence function}.

  Let $\Gamma$ be a hypergraph.
  We refer to the vertex and hyperedge sets of $\Gamma$ by $V\Gamma$ and $E\Gamma$, respectively.
  We will refer to the incidence function by simply $\iota$ when $\Gamma$ is clear from context.
  We say that a hyperedge $e \in E\Gamma$ is \emph{incident} to a vertex $v \in V\Gamma$ if $v \in \iota(e)$.
  A pair of hyperedges $e, e' \in E\Gamma$ are \emph{coincident} if $\iota(e) \cap \iota(e') \neq \emptyset$.
  Two vertices $v, v' \in V\Gamma$ are \emph{adjacent} if there is a hyperedge $e \in E\Gamma$ with $v, v' \in \iota(e)$.

  More generally, if $Y \subset V\Gamma$, we say that a hyperedge $e \in E\Gamma$ is \emph{incident} to $Y$ if $\iota(e) \cap Y \neq \emptyset$.
  Let $Y_1, \dots, Y_n, Z$ be subsets of $V\Gamma$.
  We say that a hyperedge $e \in E\Gamma$ has \emph{type} $(Y_1, Y_2, \dots, Y_n; Z)$ if $e$ is incident to each $Y_i$ for $i=1, \dots, n$ but $e$ is not incident to $Z$.
  When $Z$ is empty, we will write $(Y_1, Y_2, \dots, Y_n)$ instead of $(Y_1, Y_2, \dots, Y_n; \emptyset)$.
  We denote by $[Y_1, Y_2, \dots, Y_n; Z]_\Gamma$ the number of hyperedges of $\Gamma$ of type $(Y_1, Y_2, \dots, Y_n; Z)$.

  Let $Y \subseteq V\Gamma$, and let $Y'$ denote the complement $V\Gamma-Y$.
  We define the \emph{capacity} of $Y$ in $\Gamma$ to be the number of hyperedges of $\Gamma$ incident to both $Y$ and its complement; in the above notation,
  \begin{displaymath}
    \cp_\Gamma(Y) = [Y,Y']_\Gamma.
  \end{displaymath}

  Let $v \in V\Gamma$. The \emph{degree} of $v$ in $\Gamma$ is the number of edges incident to $v$; in the above notation,
  \begin{displaymath}
    \deg_\Gamma(v) = [v]_\Gamma.
  \end{displaymath}
\end{defn}

\begin{defn}[Whitehead hypergraph] \index{Whitehead hypergraph}
  Let $S$ be a cyclically reduced $X$-digraph.
  We define the \emph{Whitehead hypergraph of $S$} to be the hypergraph $\Gamma(S):= (X^\pm, VS, \hl:VS \rightarrow \mathcal{P}(X^\pm))$.
\end{defn}

Given a cyclically reduced $X$-digraph $S$ and an automorphism $\phi \in \Aut F(X)$, one may construct $\phi(S)$ from $S$ as follows.
First, for all $x \in X$, we subdivide every $x$-edge in $S$ into a path and relabel this path with $\phi(x)$.
We fold the resulting graph and then delete leaves until none remain; the final graph is $\phi(S)$.
When $S$ represents the conjugacy class $H^{\Aut F(X)}$, we have that $\phi(S)$ represents $\phi(H)^{\Aut F(X)}$.

When $\phi=(C,m)$ is a type II Whitehead automorphism, this construction has the special feature of being ``local''.
Let $v \in VS$ be such that $m \in \hl(v)$, and let $e$ be the $m$-edge with endpoints $v$ and $u$ for some $u \in VS$.
We ``unhook'' each edge in $\lk(v)$ with label in $C-m$ and reconnect that edge to $u$ instead.
If $v \in VS$ is such that $m \notin \hl(v)$, we then construct an \emph{auxiliary vertex} $v_{\aux}$ and an auxiliary $m$-edge with initial vertex $v_{\aux}$ and terminal vertex $v$.
We again ``unhook'' the edges of $\lk(v)$ with label in $C-\{m\}$ and reconnect them to $v_{\aux}$.
The result of performing these moves at every vertex is the graph $\phi_{\aux}(S)$, and we obtain $\phi(S)$ from $\phi_{\aux}(S)$ by cyclic reduction.
(See Figure \ref{Fig--Applying (C,m)}; the dotted edges represent edges which may or may not be present.)

\begin{figure}
  \begin{center}
    \subfigure[The neighborhood of $v$ before applying $\phi=(C,m)$.]{\includegraphics{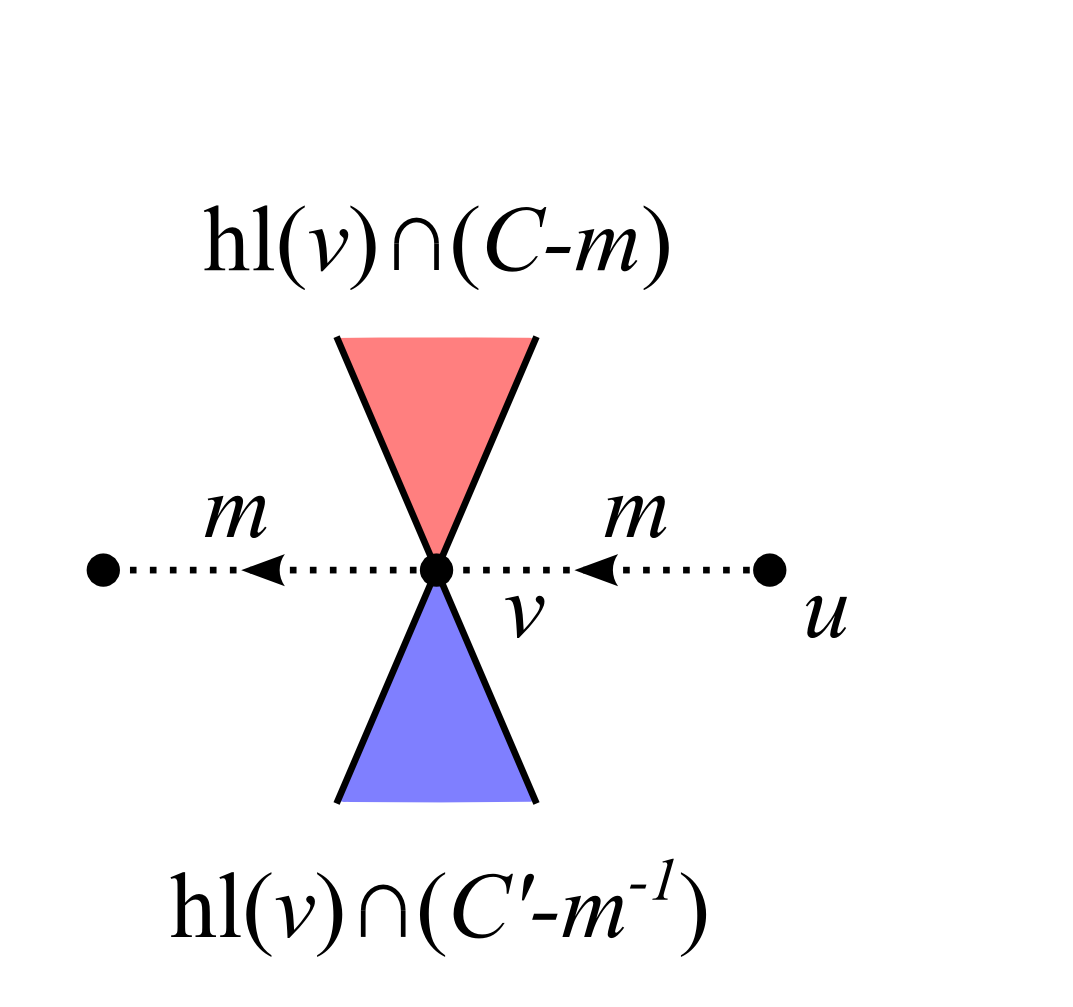}}
    \hspace{1cm}
    \subfigure[Constructing the auxiliary vertex $v_{\aux}$.]{\includegraphics{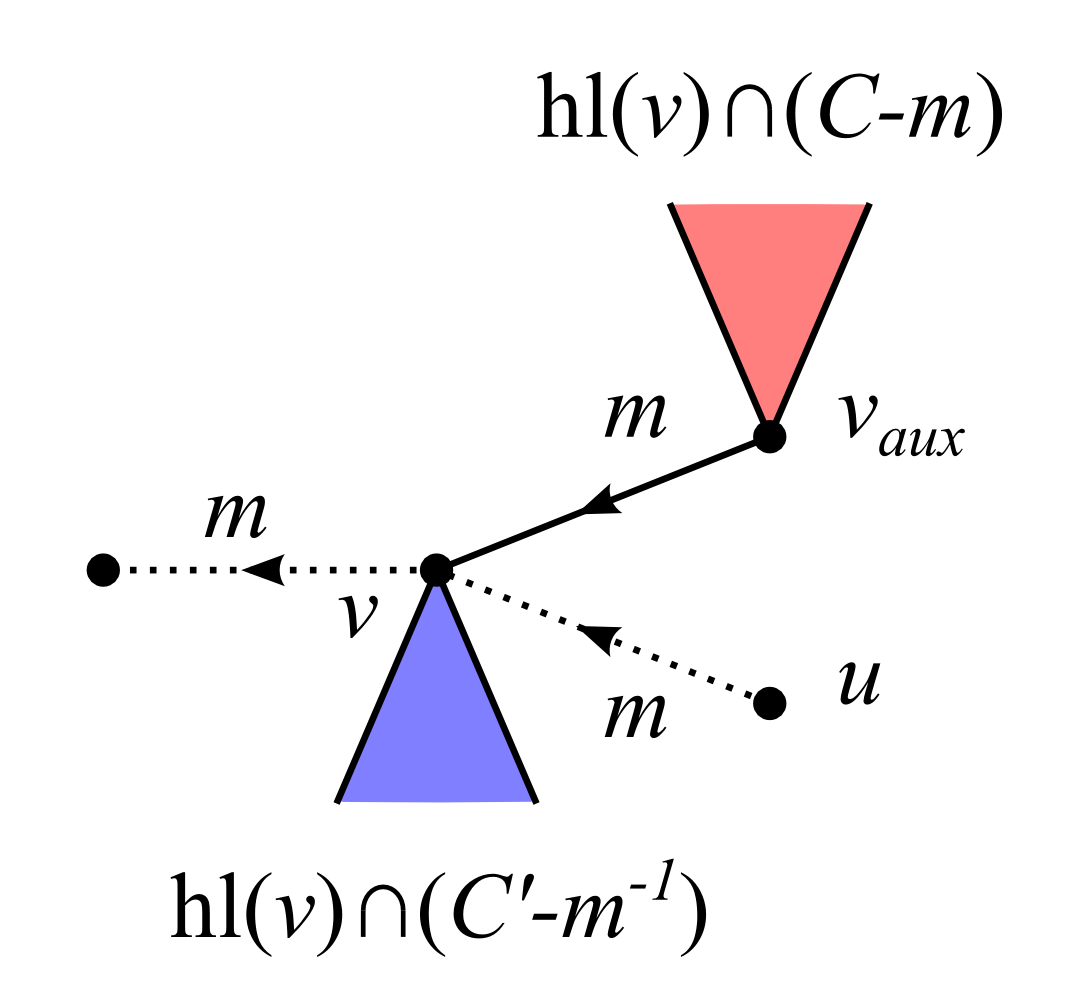}}
    \hspace{1cm}
    \subfigure[Folding identifies $v_{\aux}$ and the vertex $u$ at the other end of the edge corresponding to $m$ in $\hl(v)$. If no such vertex $u$ exists, no folding occurs.]{\includegraphics{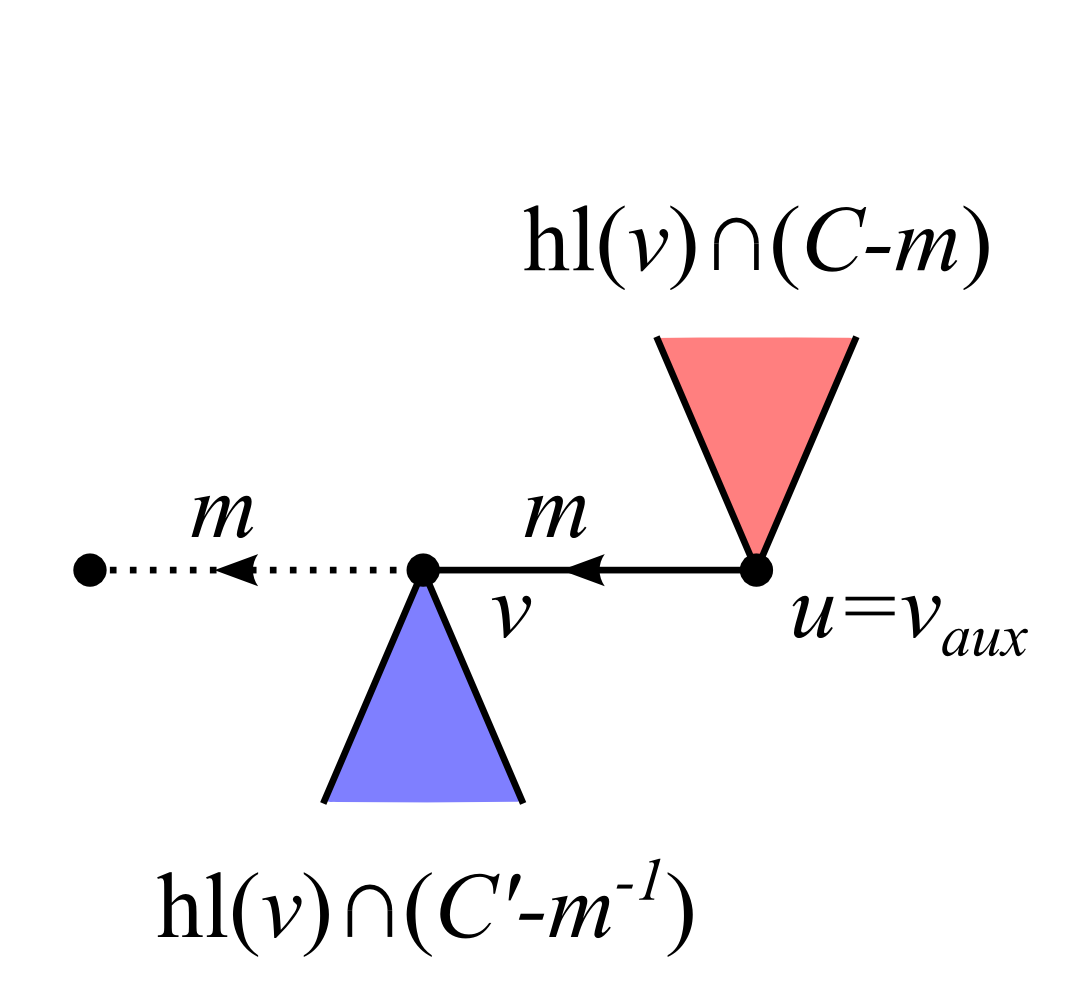}}
    \caption[The local effect of $\phi=(C,m)$ on an $X$-digraph]{Locally, the Whitehead automorphism $\phi=(C,m)$ moves the edges in $\hl(v) \cap (C-m)$ across the edge corresponding to $m \in \hl(v)$ (if present).}
    \label{Fig--Applying (C,m)}
  \end{center}
\end{figure}

We make the following observations about $\phi_{\aux}(S)$:
\begin{enumerate}
  \item
    There is an injection from the vertex set of $S$ to the set of non-auxiliary vertices of $\phi_{\aux}(S)$; we will refer to this injection simply as $\phi_{\aux}$.
  \item
    \begin{enumerate}
      \item
        Let $e$ be an $m$-edge of $S$ such that $e_- = u$ and $e_+ = v$.
        Then
        \begin{equation}
          \label{Eqn--Hyperlink 1}
          \hl(\phi_{\aux}(v)) = (\hl(v) \cap (C'\cup m) \cup (\hl(u) \cap (C - m).
        \end{equation}
      \item
        Let $v \in VS$ with $m \notin \hl(v)$. Then
        \begin{equation*}
          \label{Eqn--Hyperlink 2}
          \hl(v_{\aux}) = m^{-1} \cup \left(\hl(v) \cap \left(C - m\right)\right).
        \end{equation*}
    \end{enumerate}
  \item
    A vertex of $\phi_{\aux}(S)$ is a leaf if and only if it is one of the following:
    \begin{enumerate}
      \item
        $v_{\aux}$ for $v \in VS$ with $\hl(v) \subseteq C'$; or
      \item
        $\phi_{\aux}(v)$ for $v \in VS$ with $\hl(v) \subseteq (C-m)$.
    \end{enumerate}
\end{enumerate}

As a result, note that if $\phi=(\{m\},m)$, then $\phi(S) = S$.
If $\phi = (C,m)$ and $\phi' = (C', m^{-1})$, then note that $\phi(S) = \phi'(S)$.
This latter observation allows us to assume that, without loss of generality, $m \in X$.

By keeping careful track of the construction for $\phi(S)$, it is possible to describe the change in the number of vertices between $S$ and $\phi(S)$.

\begin{prop}[\cite{Roig2007}]
  \label{RVW--Prop--cap - deg}
  Let $S$ be a connected, cyclically reduced $X$-digraph with Whitehead hypergraph $\Gamma = \Gamma(S)$, and let $\phi=(C,m)$ be a type II Whitehead automorphism with $m \in X$.
  Then we have:
  \begin{equation*}
    \#V\phi(S) - \#VS = \cp_\Gamma(C) - \deg_\Gamma(m).
  \end{equation*}
\end{prop}

We will find it useful to recast Proposition \ref{RVW--Prop--cap - deg} in terms of change in number of edges.

\begin{prop}
  \label{Solie--Prop--(C,m) only changes m edge count}
  Let $S$ be a connected, cyclically reduced $X$-digraph with Whitehead hypergraph $\Gamma = \Gamma(S)$, and let $\phi=(C,m)$ be a type II Whitehead automorphism with $m \in X$.
  Then we have:
  \begin{enumerate}
    \item
      $\#E\phi(S) - \#ES = \cp_\Gamma(C) - \deg_\Gamma(m)$
    \item
      For a Whitehead automorphism $\phi = (C,m)$ with $m \in X$, we have
      \begin{equation*}
        \#E_m\phi(S) - \#E_m S = \cp_\Gamma(C) - \deg \Gamma(m)
      \end{equation*}
      and
      \begin{equation*}
        \#E_xS = \#E_x\phi(H)
      \end{equation*}
      for all $x \neq m$.
  \end{enumerate}
\end{prop}

\begin{proof}
  Let $S$ represent the conjugacy class $H^{\Aut F(X)}$.
  Since $S$ is connected, we have the well-known relation $\#ES = \#VS - 1 + R$, where $R$ is the rank of $H$ as a free group.
  Since $\phi(S)$ represents the class $\phi(H)^{\Aut F(X)}$ and $\phi(H)$ must also have rank $R$, we then have $\#E\phi(S) = \#V\phi(S) - 1 + R$, and part 1 follows immediately.

  Part 2 follows from the ``local'' version of the construction of $\phi(S)$.
  The only positive edges introduced in the subdivision stage have label $m$, and the only leaves which arise after subdivision and folding are leaves with hyperlink $\{m\}$.
  Therefore, the only positive edges added or removed in the application of $\phi$ to $S$ are those labeled $m$.
\end{proof}

We will recast Gersten's version of Whitehead's algorithm in graph-theoretic terms first seen in \cite{Kalajdvzievski1992} and used later in \cite{Roig2007} to analyze the complexity of the Whitehead reduction process.

Let $S$ be a connected, cyclically reduced $X$-digraph and let $\phi \in \Aut F(X)$.
We call $\phi(S)$ an \emph{automorphic image} of $S$.
We say that $\phi$ \emph{reduces} $S$ if $\#VS < \#V\phi(S)$ (or equivalently, $\#ES < \#E\phi(S)$), and that $\phi$ \emph{expands} $S$ if $\#VS > \#V\phi(S)$ (or equivalently, $\#ES > \#E\phi(S)$).
Where $S$ is clear from context, we will say that $\phi$ is \emph{reducing} or \emph{expanding}.
If the number of edges of $S$ is minimal among all its automorphic images, then we say that $S$ is \emph{minimal}.

\begin{thm}[Whitehead's Theorem \cite{Gersten1984}]
  Let $S$ be a connected, cyclically reduced $X$-digraph.
  \begin{enumerate}
    \item
      If $S$ is not minimal, then some Whitehead automorphism reduces $S$.
    \item
      Let $S$ be minimal, and suppose there is $\phi \in \Aut F(X)$ such that $\phi(S)$ is also minimal.
      Then there exists a sequence of type II Whitehead automorphisms $\phi_1, \dots, \phi_k$ such that $\phi_i$ does not expand $\phi_{i-1} \circ \dots \circ \phi_1(S)$ and $\phi_k \circ \dots \circ \phi_1 (S) = \phi(S)$.
  \end{enumerate}
\end{thm}

Let $S$ be a cyclically reduced $X$-digraph.
Let $\min(S)$ denote the set of minimal automorphic images of $S$.
Whitehead's Theorem gives an effective algorithm for constructing $\min(S)$ given $S$.
Let $S$ and $T$ be cyclically reduced $X$-digraphs representing conjugacy classes $H^{\Aut F(X)}$ and $K^{\Aut F(X)}$; then there exists $\phi \in \Aut F(X)$ such that $K = \phi(H)$ if and only if $\min(S) = \min(T)$.
This is Gersten's extension of Whitehead's famous algorithm to finitely generated subgroups.

\begin{thm}[Whitehead's algorithm]
  There is an algorithm to decide, given $H, K \fg F(X)$, whether or not there exists $\phi \in \Aut F(X)$ such that $\phi(H)=K$.
\end{thm}

\subsection{Splittings of Free Groups}

Bass-Serre theory broadly concerns the structural implications of free group actions on trees. We will not need the full extent of the theory here, though we record some of the terminology here for completeness. An introduction to the theory may be found in \cite{Serre2003}.

\begin{defn}[Cyclic splitting] \index{cyclic splitting}
  \label{Defn--Cyclic splitting}
  A \emph{cyclic splitting} of $F(X)$ is the decomposition of $F(X)$ as the fundamental group of a graph of groups with cyclic edge groups.
  A \emph{free splitting} of $F(X)$ is the decomposition of $F(X)$ as the fundamental group of a graph of groups with trivial edge groups.
  An \emph{edge map} refers to a homomorphism from an edge group to a vertex subgroup in a particular graph of groups.
  A splitting is \emph{elementary} if the corresponding graph of groups is connected and has exactly one edge.
  An elementary splitting is a \emph{segment splitting} if the underlying graph of groups has two distinct vertices and is a \emph{loop splitting} if it has only one vertex.
  An elementary splitting is \emph{nontrivial} if it is either a loop splitting or a segment splitting in which neither edge map is an isomorphism.
  An elementary cyclic splitting is \emph{very small} if the image of the edge group is maximal cyclic in the vertex subgroup(s).

  We say that $H \fg F(X)$ is \emph{elliptic} in a splitting of $F(X)$ if $H$ is conjugate to a subgroup of a vertex subgroup.
  Subgroups which are not elliptic in a given splitting are said to be \emph{hyperbolic}.
\end{defn}

\begin{prop}
  \label{Solie--Prop--Characterization of vertex subgroups}
  The vertex subgroups in a nontrivial, very small, elementary cyclic segment splitting of $F(X)$ have the form
  \begin{equation*}
    \langle A, b \rangle \text{ and } \langle B \rangle,
  \end{equation*}
  where $A \sqcup B$ is a basis for $F(X)$, $\#A \geq 1, \#B \geq 2$, and $b \in \langle B \rangle$ is not a proper power.

  The vertex subgroup in a cyclic loop splitting of $F(X)$ has the form
  \begin{equation*}
    \langle U, u^v \rangle,
  \end{equation*}
  where $U \sqcup \{v\}$ is a basis for $F(X)$ and $u \in \langle U \rangle$ is not a proper power.
\end{prop}

\begin{figure}
  \begin{center}
    \subfigure[A standard segment vertex subgroup.]{\includegraphics[scale=0.75]{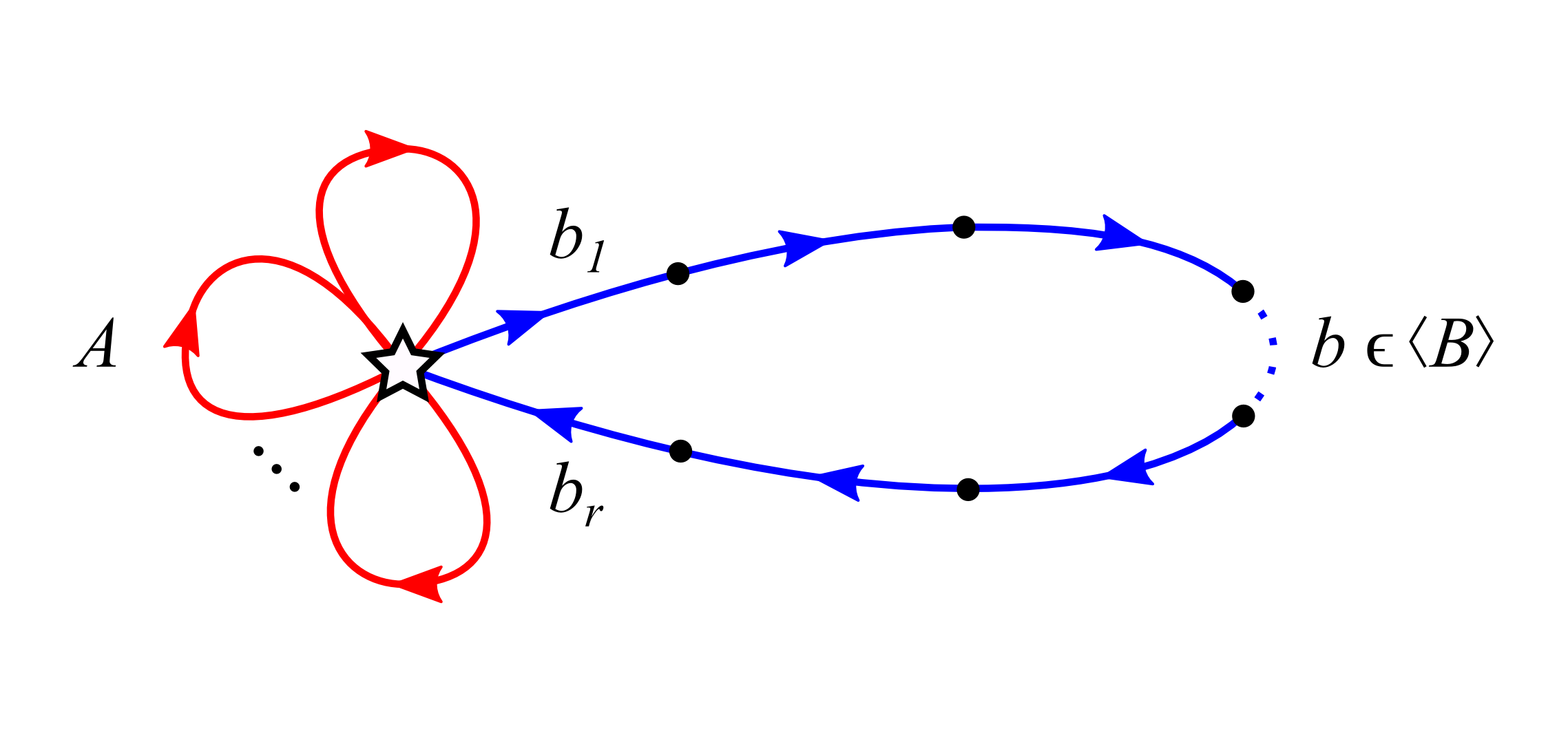}}
    \subfigure[A standard loop vertex subgroup.]{\includegraphics[scale=0.75]{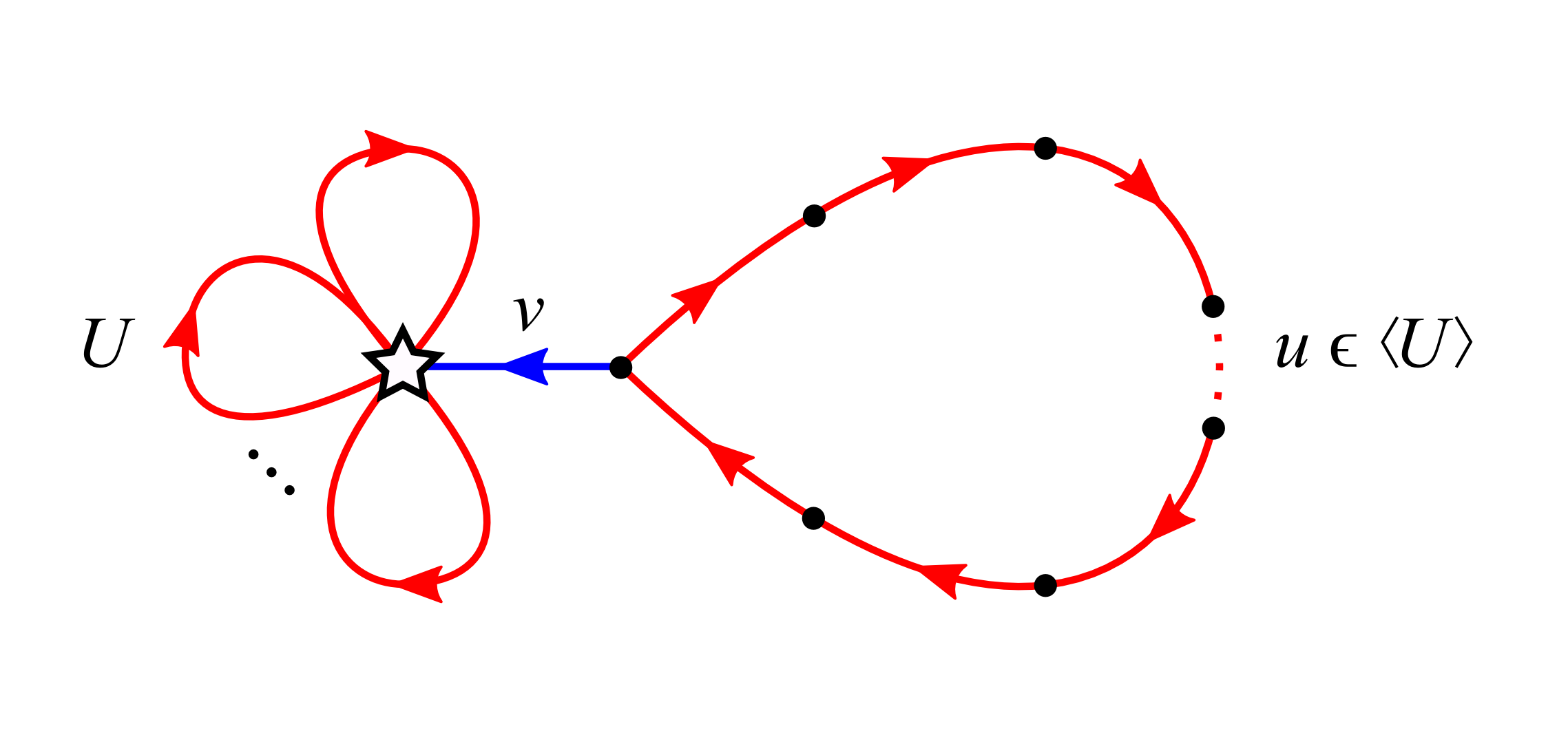}}
    \caption{Stallings graphs of standard vertex subgroups.}
    \label{Fig--Standard vertex subgroups}
  \end{center}
\end{figure}

\begin{proof}
  This is a straightforward application of a lemma of Bestvina-Feighn \cite[Lemma 4.1]{Bestvina1994}.
  Similar results also appear in \cite{Shenitzer1955,Stallings1980,Swarup1986}.
\end{proof}

\begin{defn}[Segment, loop vertex subgroups] \index{vertex subgroup!segment} \index{vertex subgroup!loop}
  \label{Defn--Segment, loop vertex subgroups}
  We call a subgroup $\langle A, b \rangle$  as in Proposition \ref{Solie--Prop--Characterization of vertex subgroups} a \emph{segment vertex subgroup}.
  A subgroup $\langle U, u^v \rangle$ is called a \emph{loop vertex subgroup}.
  When $A \sqcup B = X$ or $U \sqcup \{v\} = X$, we say that these vertex subgroups are \emph{standard}.
  By $\mathcal{SV}$ and $\mathcal{LV}$ we denote the sets of standard segment and standard loop vertex subgroups, respectively.
\end{defn}

If $H$ is a segment vertex subgroup, then for the automorphism $\phi$ induced by a bijection $A \sqcup B \rightarrow X$, $\phi(H) \in \mathcal{SV}$.
Thus, every segment vertex subgroup has an automorphic image in the set $\mathcal{SV}$.
Likewise, every loop vertex subgroup has an automorphic image in $\mathcal{LV}$ and every proper free factor has an automorphic image in $\mathcal{SF}$.

Let $H \fg F(X)$ be a proper free factor of $F(X)$.
If $H = \langle Y \rangle$ where $Y \subset X$, then we say that $H$ is a \emph{standard free factor}.
By $\mathcal{SF}$ we denote the set of standard free factors of $F(X)$.
Note that every standard free factor is a subgroup of a standard loop vertex subgroup.

\section{Main Results}

For convenience, we define the following general problem.

\begin{defn}[Automorphic orbit problem] \index{Automorphic orbit problem}
  \label{Defn--Automorphic orbit problem}
  Let $\mathcal{K}$ be a (possibly infinite) collection of subgroups of $F(X)$. The \emph{automorphic orbit problem}, denoted $\AO(\mathcal{K})$, is the problem:
  \bigskip
  \begin{center}
    \parbox[c]{4in}{Given $H \fg F(X)$, do there exist $K \in \mathcal{K}$ and $\phi \in \Aut F(X)$ such that $\phi(H) \leq K$?}
  \end{center}
  \bigskip
\end{defn}

In the case where $\mathcal{H}$ consists of a single cyclic subgroup, $\AO(\mathcal{H})$ is solved by Whitehead's algorithm.

\subsection{Proper Free Factors}

Recall that $\mathcal{SF}$ is the set of standard free factors of $F(X)$, and that $H \fg F(X)$ is contained in a proper free factor of $F(X)$ if and only if $H$ has some automorphic image which is a subgroup of a standard free factor.

\begin{prop}
   A subgroup $H \fg F(X)$ is contained in a proper free factor of $F(X)$ if and only if each element of $\min(S_X(H))$ omits some letter of $X$ from its set of edge labels.
\end{prop}

\begin{proof}
  Let $H \fg F(X)$.
  Let $Y$ be a basis for $F(X)$ such that $H \leq \langle Y' \rangle$ for some $Y' \subset Y$.
  Let $\phi \in \Aut F(X)$ be induced by a bijection $Y \rightarrow X$, so that $\phi(Y'):= X' \subset X$.
  The graph $S=S_X(\phi(H))$ therefore is in the automorphic orbit of $S_X H$ and omits some element $m \in X$ as an edge label.

  Let $\psi=(C,m)$ be a Whitehead automorphism, where $S$ omits $m$ as an edge label.
  Proposition \ref{Solie--Prop--(C,m) only changes m edge count} states that applying $\psi$ to $S$ changes only the number of positive edges labeled $m$, and so $\psi$ cannot be reducing.
  
  We conclude that any reducing Whitehead automorphism for $S$ must have a multiplier which occurs as an edge label in $S$.
  Therefore if the $X$-digraph $S$ omits $m$ as an edge label, $S$ can be minimized by a sequence of Whitehead automorphisms, each with multiplier different from $m$.
  After minimizing, we see that some $S_0 \in \min(S_X(H))$ must therefore omit some letter of $X$ from its set of edge labels.

  The second part of Whitehead's algorithm states that \emph{every} element of $\min(S_X(H))$ can be accessed from $S_0$ by a sequence of non-expanding Whitehead automorphisms.
  This implies that all elements of $\min(S_X(H))$ must omit at least one element of $X$ from its set of edge labels.

  Suppose otherwise; let $S' \in \min(S_X(H))$ have all letters of $X$ appear as edge labels.
  There must be a sequence of Whitehead automorphisms $\phi_1, ..., \phi_n$ such that $S' := S_n = \phi_n \circ ... \circ \phi_1(S_0)$ and $\phi_{i+1}$ does not expand $S_i:=\phi_i \circ ... \circ \phi_1(S_0)$ for $i=1,2,...,n$.
  However, since $S'$ has all elements of $X$ appearing as edge labels, there must be a least index $k$ such that all letters of $X$ appear as labels of $S_{k}$.
  Suppose that $S_{k-1}$ omits the letter $x \in X$ from its set of edge labels.
  As $\phi_{k}(S_{k-1}) = S_k$ and $x$ appears in the edge labels of $S_k$, the Whitehead automorphism $\phi_k$ must be type II with multiplier $x$.
  Since applying $\phi_k$ changes only the number of edges labelled $x$ and $S_{k-1}$ has no such edges, $\phi_k$ must be expanding, a contradiction.

  Conversely, it is straightforward to see that if some (every) element of $\min(S_X(H))$ omits a letter from $X$, then $H$ is contained in a proper free factor.
\end{proof}

The following result is well-known, and we use the above machinery to prove it for completeness.

\begin{cor}
  The problem $\AO(\mathcal{SF})$ is decidable.

\end{cor}

\begin{proof}
  The subgroup $H$ is contained in a proper free factor if and only if there exist $\phi \in \Aut F(X)$ and $K \in \mathcal{SF}$ such that $\phi(H) \leq K$.
  However, $\phi(H) \leq K \in \mathcal{SF}$ if and only if some (every) element of $\min(S_X(\phi(H)))$ omits an element of $X$ as an edge label.
  Since $\min(S_X(\phi(H))) = \min(S_X(H))$, our algorithm is as follows.

  \begin{alg}
    Given $H \fg F(X)$, we may determine whether or not $H$ is contained in a proper free factor of $F(X)$ as follows:
    \begin{enumerate}
      \item
        Construct the finite graph $S_X(H)$.
      \item
        Construct an element $T$ of $\min(S_X(H))$.
      \item
        Determine whether $T$ omits some element of $X$ as an edge label.
        \begin{enumerate}
          \item If $T$ omits some element of $X$ as an edge label, conclude that $H$ is contained in a proper free factor of $F(X)$.
          \item Otherwise, conclude that $H$ is contained in no proper free factor of $F(X)$.
        \end{enumerate}
    \end{enumerate}
  \end{alg}

  Note that by inverting the specific sequence of Whitehead automorphisms used to construct the minimal element $T$ and applying it to the standard free factor containing $T$, we may construct the explicit free factor of $F(X)$ containing $H$. This solves the associated search problem.
\end{proof}

\subsubsection{Segment vertex subgroups in higher rank}

Let $F(X)$ be a free group of rank at least three.

Recall that $\mathcal{SV}$ is the set of subgroups of $F(X)$ of the form $\langle A, b \rangle$ where $A \sqcup B = X$, $\#A \geq 1$, $\#B \geq 2$, and $b \in \langle B \rangle$ is not a proper power.

Let $S$ be an $X$-digraph and let $Y \subseteq X$.
The subgraph of $S$ \emph{spanned} by the $Y$-edges is the subgraph consisting of all $Y$-edges and all vertices having an incident $Y$-edge.

\begin{defn}[Property $(S)$] \index{Property $(S)$}
  \label{Defn--Property (S)}
  We say that an $X$-digraph \emph{satisfies property $(S)$} if:
  \begin{enumerate}
    \item
      $S$ is connected and cyclically reduced; and
    \item
      There is a partition $X = A \sqcup B$ such that:
      \begin{enumerate}
        \item
          The subgraph spanned by the $A$-edges is a bouquet of single-edge loops; and
        \item
          The subgraph spanned by the $B$-edges is rank one.
      \end{enumerate}
  \end{enumerate}
\end{defn}

Any element of $\mathcal{SV}$ has a Stallings graph which satisfies Property $(S)$.
We immediately obtain the following.

\begin{prop}
  \label{Solie--Prop--Immersion Property (S)}
  Let $H \fg F(X)$.
  Then $H$ is a subgroup of some element of $\mathcal{SV}$ if and only if $S_X(H)$ admits an immersion into a graph satisfying Property $(S)$.
\end{prop}

\begin{figure}
  \begin{center}
    \includegraphics{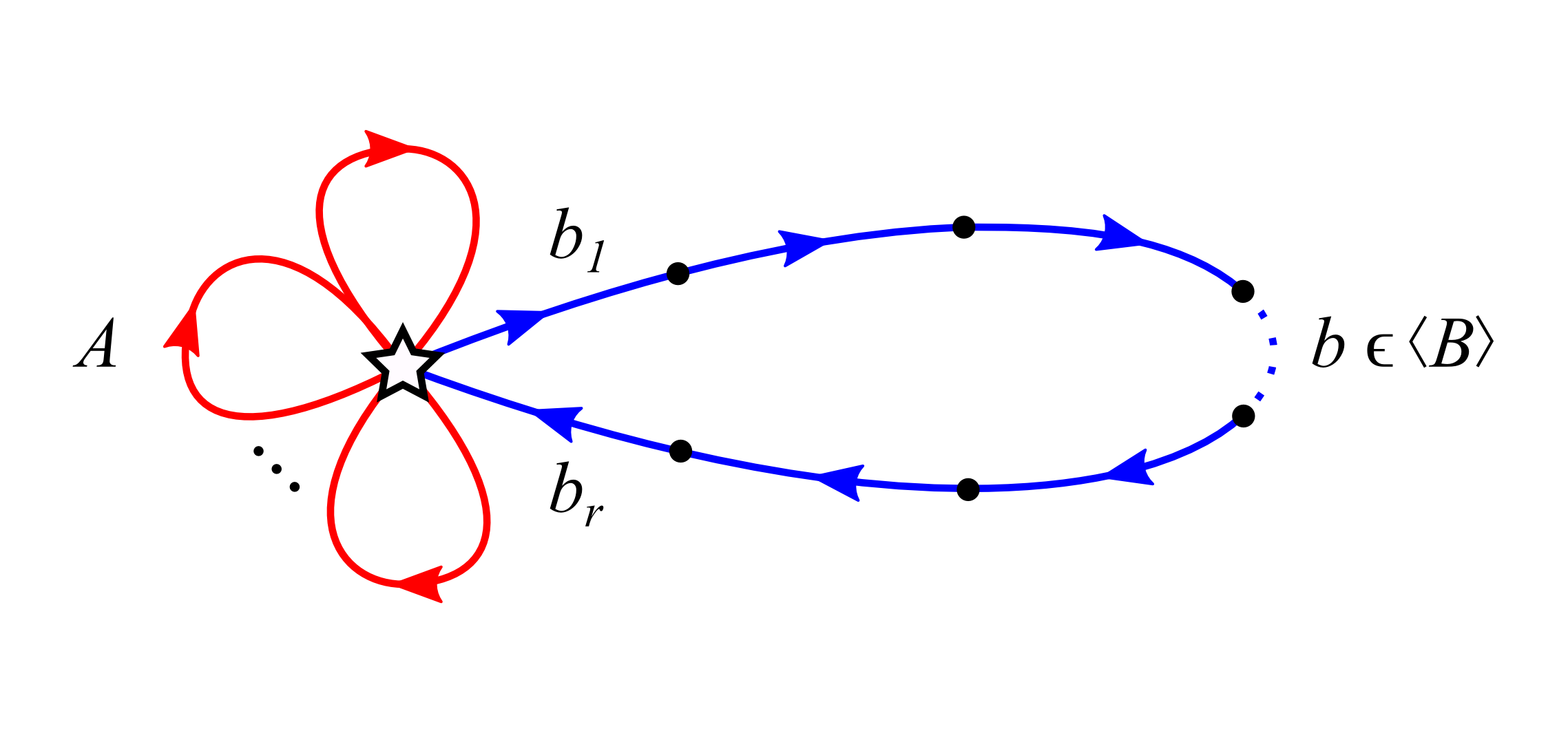}
    \caption{Stallings graph of a standard segment vertex subgroup.}
  \end{center}
\end{figure}

\begin{lem}
  Let $S$ be a connected, cyclically reduced $X$-digraph admitting an immersion onto a graph satisfying property $(S)$.
  Suppose that $S$ represents (the conjugacy class of) a subgroup contained in no proper free factor of $F(X)$.
  If $S$ is not minimal, then some element of $\min(S)$ also admits an immersion onto a graph satisfying property $(S)$.
\end{lem}

\begin{proof}
  Suppose $T$ is an $X$-digraph satisfying Property $(S)$ such that $\pi:S \rightarrow T$ is an immersion.
  Let $A \sqcup B = X$ be the partition given in the definition of Property $(S)$.
  Let the \emph{basepoint} of $T$ be the unique vertex whose hyperlink meets $A^\pm$, and let $b = b_1 \dots b_r$ be the label of the loop in $T$ labeled by $B$-edges, beginning and ending at the basepoint.
  Note that the hyperlink of the basepoint is $A^\pm \cup \{b_1^{-1}, b_r\}$.

  Since $\pi: S \rightarrow T$ is an immersion, there is a $k$ such that, for any non-basepoint $v \in VT$, the preimage $\pi^{-1}(v)$ is a set of exactly $k$ vertices.
  More, the subgraph of $S$ spanned by the $B$-edges is the union of exactly $k$ paths labeled by $b$, any two of which are either disjoint or intersect only at one or both endpoints.

  The following technical proposition will provide useful sufficient conditions for Property $(S)$ to be preserved.

  \begin{prop}
    \label{Solie--Prop--Sufficient condition for preservation of (S)}
    \mbox{}
    \begin{enumerate}
      \item
        Let $\phi = (C,m)$ be a Whitehead automorphism with $m \in A^\pm$ and let $T$ be as above.
        If $C$ does not cut the hyperlink of any non-basepoint vertex of $T$, then $\phi(T)$ also satisfies property $(S)$.
      \item
        Let $\phi = (C,m)$ be a Whitehead automorphism with $m \in B^\pm$ and let $T$ be as above.
        If $C$ does not cut the set $A^\pm$, then $\phi(T)$ also satisfies property $(S)$.
    \end{enumerate}
  \end{prop}

  \begin{proof}
    Suppose $\phi = (C,m)$ is as in part 1 of the proposition.
    In the construction of $\phi(S)$, new $m$-edges are only introduced at vertices of $S$ whose hyperlinks are cut by $C$.
    Therefore, the only new edges introduced in the application of $\phi$ are incident to the basepoint; since every $A$-edge of $T$ has the basepoint as its initial and terminal vertex, these new edges are folded away, leaving a $B$-labeled loop beginning and ending at the basepoint.

    Suppose $\phi = (C,m)$ is as in part 2 of the proposition.
    If $A^\pm$ is not cut by $C$, then the effect of $\phi$ on $T$ is to replace the loop labeled $b$ with a loop labeled $\phi(b)$ or possibly $\phi(b)^{m^{-1}}$.
    The resulting graph satisfies Property $(S)$.
  \end{proof}

  \begin{conv}
    To simplify notation, we define the following sets.
    \begin{itemize}
      \item
        $\Delta := (A^\pm \cap C) - \{m, m^{-1}, b_1^{-1}, b_r\}$
      \item
        $\Sigma := (A^\pm \cap C') - \{m, m^{-1}, b_1^{-1}, b_r\}$
      \item
        $\Pi := (B^\pm \cap C) - \{m, m^{-1}, b_1^{-1}, b_r\}$
      \item
        $\Omega := (B^\pm \cap C') - \{m, m^{-1}, b_1^{-1}, b_r\}$
    \end{itemize}
    Note that we do not necessarily have that $m, m^{-1}, b_1^{-1},$ and $b_r$ are pairwise distinct.
  \end{conv}

  Suppose that $b_1^{-1} = b_r$, and consider $\Gamma(S)$.
  Since in $\Gamma(S)$, the only element of $B^\pm$ adjacent to some element of $A^\pm$ is $b_r$, a direct calculation shows that the Whitehead automorphism $\phi=(B^\pm - {b_r^{-1}}, b_r)$ reduces $S$.
  By Proposition \ref{Solie--Prop--Sufficient condition for preservation of (S)}, $\phi(T)$ satisfies Property $(S)$.
  We will therefore assume from now on that $b_1^{-1} \neq b_r$.

  Suppose that $\phi = (C,m)$ reduces $S$, where $m \in A^\pm$.
  First note that if $C$ does not cut $\{b_1^{-1},b_r\}$, then either $(C \cup B^\pm,m)$ or $(C - B^\pm, m)$ is reducing for $S$, since only $b_1^{-1}$ and $b_r$ are adjacent to elements of $A^\pm$ in $\Gamma(S)$.
  By Proposition \ref{Solie--Prop--Sufficient condition for preservation of (S)}, the image of $T$ under either of these Whitehead automorphisms again satisfies Property $(S)$.

  Now assume that, without loss of generality, $b_r \in C$ and $b_1^{-1} \in C' := X^\pm - C$ and that $C$ cuts the hyperlink of some non-basepoint vertex of $T$.
  Since the preimage under $\pi$ of a non-basepoint vertex is a set of $k$ internal vertices in $S_X(H)$, we have $[\Pi \cup \{b_r\}, \Omega \cup \{b_1^{-1}\}]_{\Gamma(S)} \geq k$.
  Therefore, passing from $(\Delta \cup \{m, b_r\} \cup \Pi,m)$ to $(\Delta \cup m,m)$ reduces the capacity by at least $k$ (since at least $k$ hyperedges contributing to capacity came from the hyperlink of an internal vertex) at the cost of adding $[b_r, \Delta \cup m]_{\Gamma(S)}$ to the capacity.
  However, a $b_r$-edge is coincident to an $A$-edge in at most $k$ vertices of $S$, so $[b_r, \Delta \cup m]_{\Gamma(S)} \leq k$.
  Therefore, $\cp_{\Gamma(S)}(\Delta \cup m) \leq \cp_{\Gamma(S)}(\Delta \cup \{m, b_r\} \cup \Pi)$, and so $\phi'=(\Delta \cup m,m)$ must reduce $S$.
  Again, by Proposition \ref{Solie--Prop--Sufficient condition for preservation of (S)}, $\phi'(T)$ satisfies Property $(S)$.
  (See Figure \ref{Fig--A multiplier}.)

  \begin{figure}
    \begin{center}
      \subfigure[$\Gamma(S)$ with $(\Delta \cup m \cup \Pi,m)$ reducing, $m \in A^\pm$.]{\includegraphics{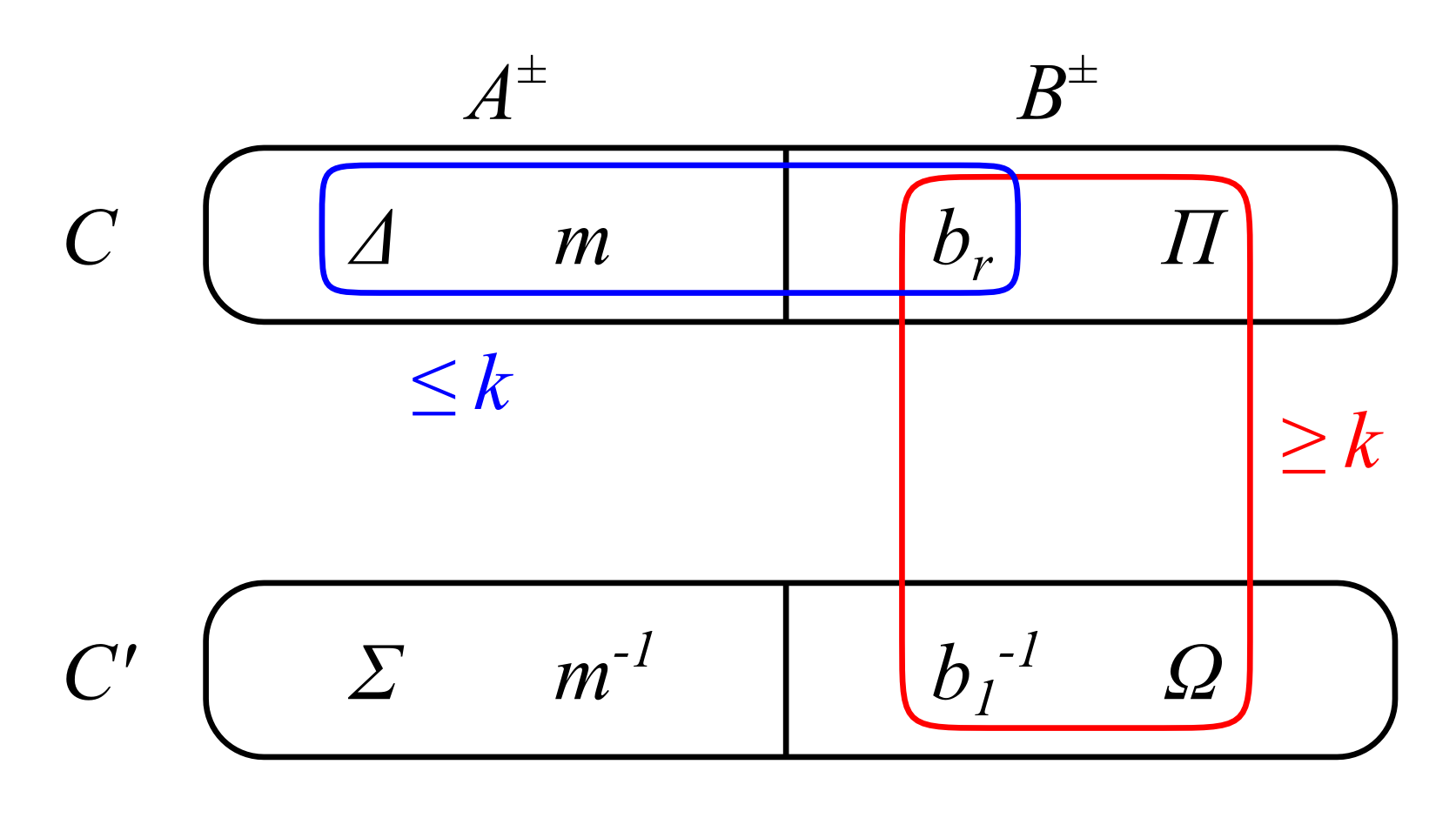}}
      \subfigure[$\Gamma(S)$ with $(\Delta \cup m,m)$ reducing, $m \in A^\pm$.]{\includegraphics{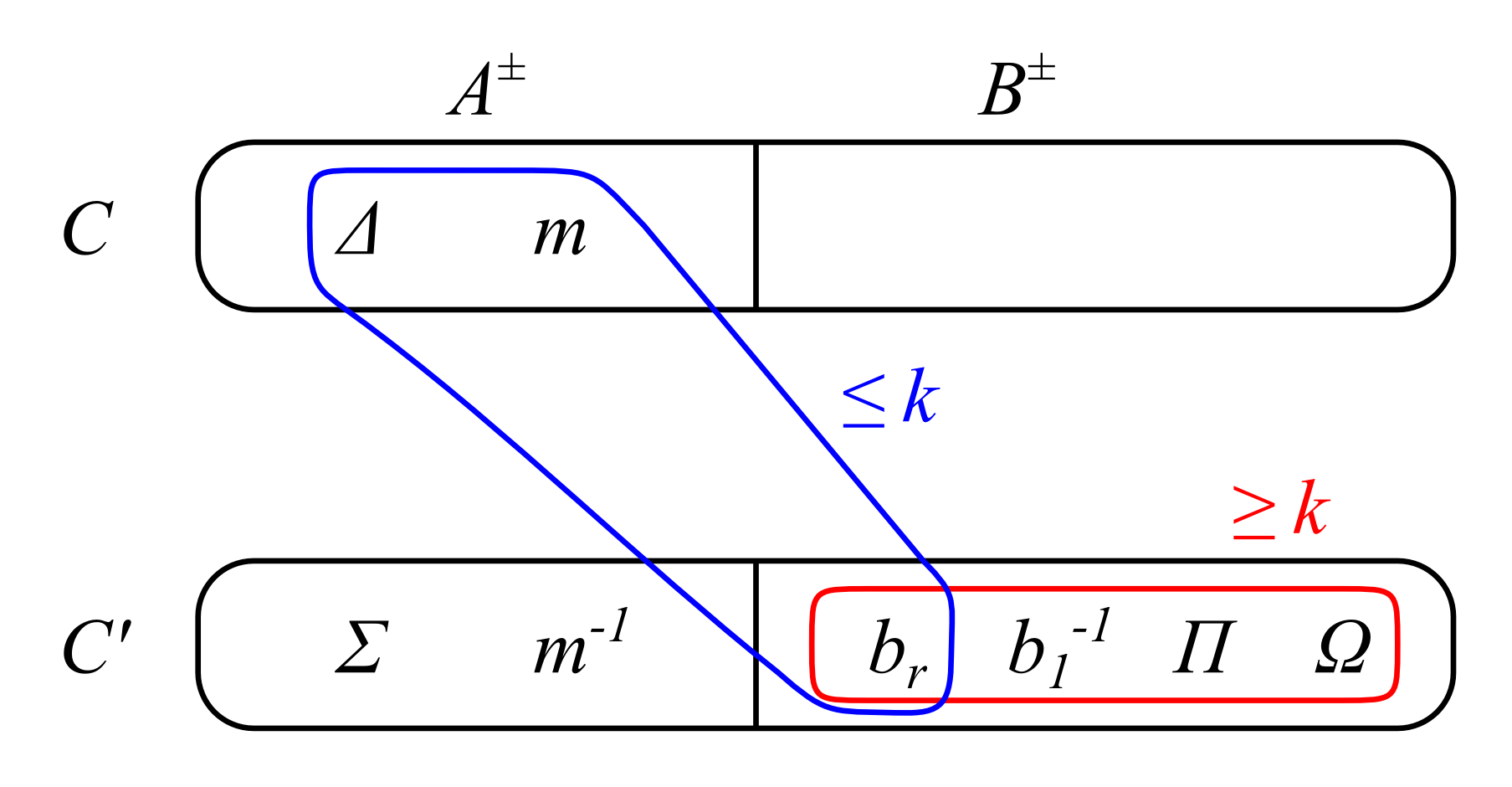}}
      \caption{If $(\Delta \cup m \cup \Pi,m)$ reduces $S$ with $m \in A^\pm$, then so must $(\Delta \cup m,m)$.}
      \label{Fig--A multiplier}
    \end{center}
  \end{figure}

  Now suppose that $\phi=(\Delta \cup \{m, b_r\} \cup \Pi,m)$ reduces $S$, where $m \in B^\pm$.
  Once again, if $C$ does not cut $\{b_1^{-1},b_r\}$, then either Whitehead automorphism $(\Delta \cup \Sigma \cup m \cup \Pi, m)$ or $(m \cup \Pi,m)$ will also reduce $S$.
  By Proposition \ref{Solie--Prop--Sufficient condition for preservation of (S)}, each of these Whitehead automorphisms preserve Property $(S)$.
  We may therefore assume that $b_r \in C$ and $b_1^{-1} \in C'$.

  Consider the quantities $[\Delta, b_r; \Sigma \cup \Omega \cup \{b_1^{-1}, m^{-1}\}]_{\Gamma(S)}$ and $[\Delta, \Sigma \cup \Omega \cup \{b_1^{-1}, m^{-1}\}; b_r]_{\Gamma(S)}$.
  Note first that
  \begin{align*}
    [\Delta, \Sigma \cup \Omega \cup \{b_1^{-1}, m^{-1}\}; b_r]_{\Gamma(S)} = [\Delta, \Sigma \cup \{b_1^{-1}\}; b_r]_{\Gamma(S)}
  \end{align*}
  because no element of $\Delta \subseteq A^\pm$ is coincident to $m^{-1}$ or $\Omega$.
  
  Suppose that 
  \begin{align*}
    [\Delta,b_r;\Sigma \cup \Omega \cup \{b_1^{-1}, m^{-1}\}]_{\Gamma(S)} \leq [\Delta, \Sigma \cup \{b_1^{-1}\}; b_r]_{\Gamma(S)}.
  \end{align*}
  Moving $\Delta$ into $C'$ must therefore not increase the capacity of the cut, and so $\phi'=(\{m, b_r\} \cup \Pi,m)$ must be reducing for $S$.
  Since $\phi'$ has multiplier in $B^\pm$ and no longer cuts $A^\pm$, its image has Property $(S)$.

  On the other hand, let 
  \begin{align*}
  [\Delta,b_r; \Sigma \cup \Omega \cup \{b_r^{-1}, m^{-1}\}]_{\Gamma(S)} > [\Delta, \Sigma \cup \Omega \cup \{b_1^{-1}, m^{-1}\}; b_r]_{\Gamma(S)}.
  \end{align*}
  As $\deg_{\Gamma(S)}(b_r)$ is at least as large as $[\Delta,b_r; \Sigma \cup \Omega \cup \{b_r^{-1}, m^{-1}\}]_{\Gamma(S)}$, we immediately see that changing the multiplier to $b_r$ yields a reducing automorphism. Thus the Whitehead automorphism $\phi'=(\Delta \cup b_r, b_r)$ reduces $S$. (See Figure \ref{Fig--B multiplier}.)
  
  Unfortunately, $\phi' = (\Delta \cup b_r, b_r)$ will almost certainly fail to preserve Property $(S)$ in most cases.
  However, when $\phi'$ is reducing, we can follow it with a specific sequence of Whitehead automorphisms whose composition effectively multiplies $\Delta$ by the entire word $b=b_1 \dots b_r$.
  The next proposition asserts that each individual Whithead automorphism in this sequence is reducing, and that Property $(S)$ is restored at the end of the sequence.

  \begin{figure}
    \begin{center}
      \subfigure[$\Gamma(S)$ with $(C,m)$ reducing, $m \in B^\pm$.]{\includegraphics{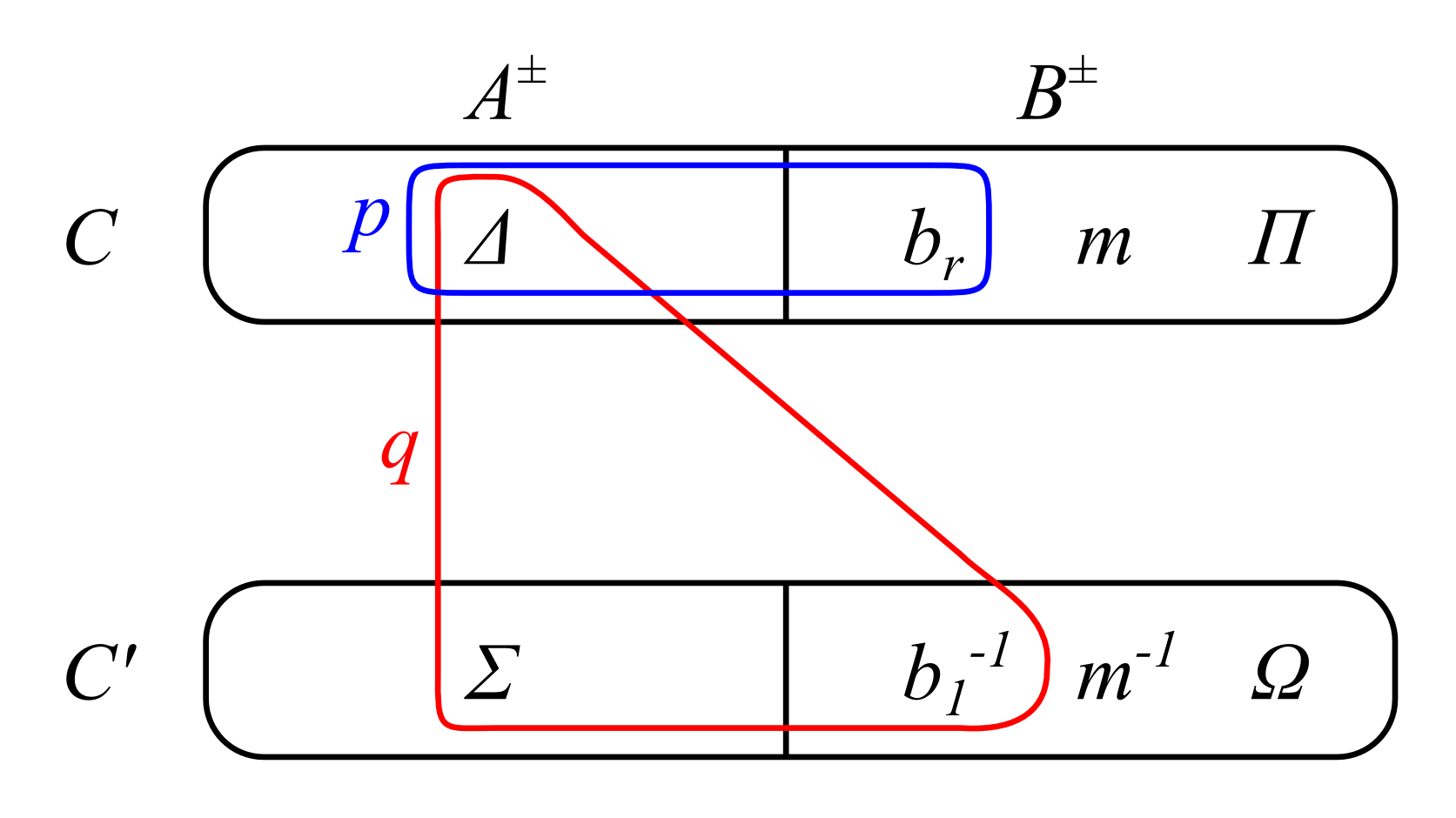}}
      \subfigure[If $p \leq q$, then $(\{b_r,m\} \cup \Pi, b_r)$ is reducing.]{\includegraphics{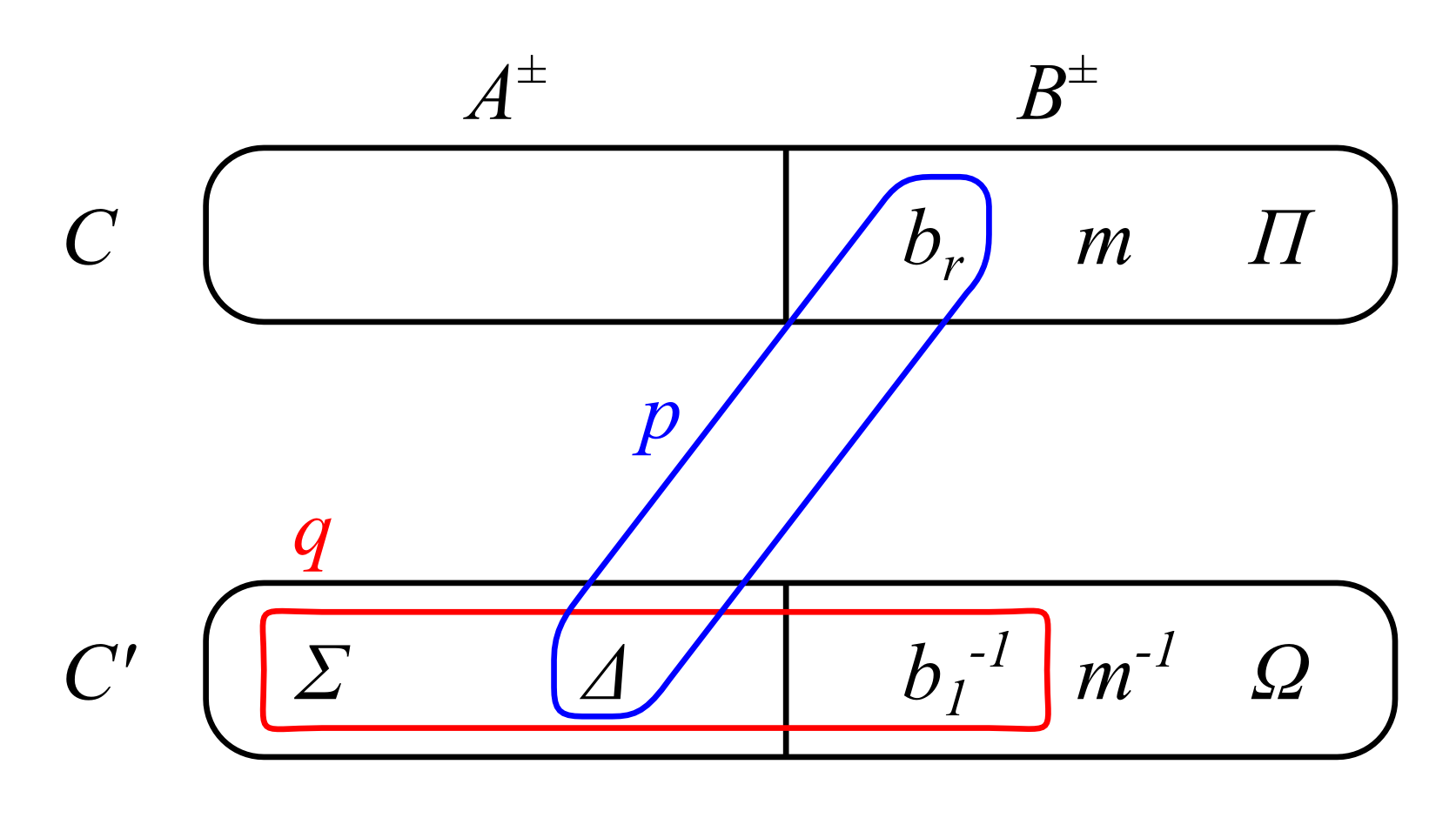}}
      \subfigure[If $p > q$, then $(\Delta \cup b_r, b_r)$ is reducing.]{\includegraphics{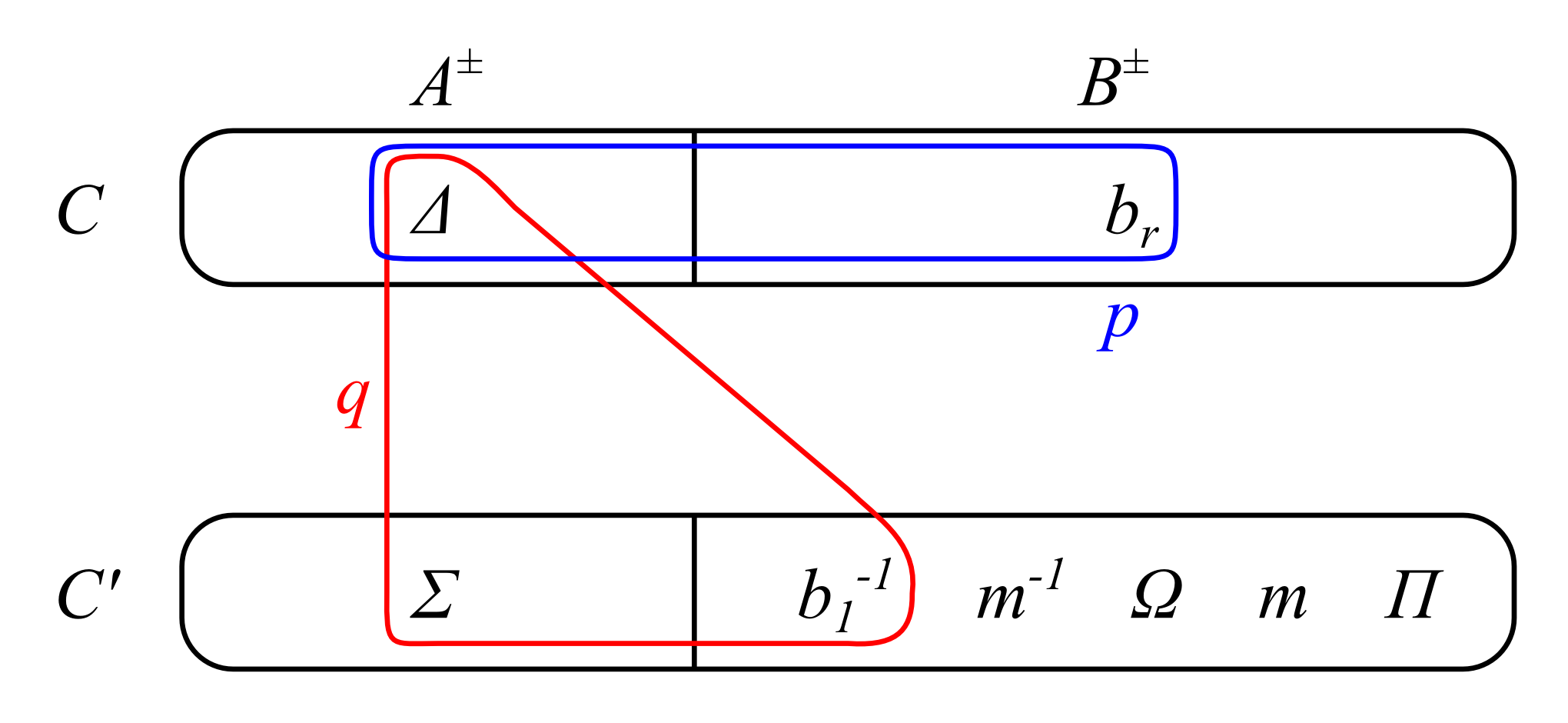}}
    \end{center}
    \caption{If $(C,m)$ reduces $S$ with $m \in B^\pm$, then either $(C\cap B^\pm,m)$ or $(\Delta \cup b_r,b_r)$ also reduces $S$.}
  \label{Fig--B multiplier}
  \end{figure}

  \begin{prop}
    \label{Solie--Prop--multiply by b}
    For $i=1, \dots, r$, define $\phi_i := (\Delta \cup b_i, b_i)$.
    If $\phi_r$ reduces $S$, then for all $i=1, \dots, r-1$,
    the Whitehead automorphism $\phi_i$ reduces $\phi_{i+1}\cdots \phi_r(S)$.
    Furthermore, $\phi_1 \dots \phi_r(S)$ immerses onto a graph satisfying Property (S).
  \end{prop}

  \begin{proof}
    First notice that since $S$ immerses onto $T$ with property $(S)$, then $\phi_{i+1}\cdots \phi_r(S)$ immerses onto $\phi_{i+1}\cdots \phi_r(T)$.
    Set $\Pi_i := B^\pm - b_i$ and $\Pi_i^{-1} := B^\pm - b_i^{-1}$.

    Suppose $v \in VS$ has hyperlink type $(\Delta, b_r; \Sigma \cup \Pi_r)$.
    The image of the vertex adjacent to $v$ via the $b_r$ edge will then have hyperlink type $(\Delta, b_{r-1}; \Sigma \cup \Pi_{r-1})$ in $\phi_r(S)$.
    Moreover, this is the only way in which a vertex of $\phi_r(S)$ may have type $(\Delta, b_{r-1}; \Sigma \cup \Pi_{r-1})$.
    Therefore
    \begin{equation*}
      [\Delta, b_r; \Sigma \cup \Pi_r]_{\Gamma(S)} = [\Delta, b_{r-1}; \Sigma \cup \Pi_{r-1}]_{\Gamma(\phi_r(S))}.
    \end{equation*}

    Suppose $v \in VS$ has type $(\Delta, \Sigma \cup \Pi_r; b_r)$.
    Then $v$ contributes an auxiliary vertex with hyperlink of type $(\Delta, b_r^{-1}; \Sigma \cup \Pi_r^{-1})$.
    Again, the only way a hyperlink of type $(\Delta, b_r^{-1}; \Sigma \cup \Pi_r^{-1})$ may arise is as such an auxiliary vertex, so
    \begin{equation*}
      [\Delta, b_r^{-1}; \Sigma \cup \Pi_r^{-1}]_{\Gamma(\phi_r(S))} = [\Delta, \Sigma \cup \Pi_r; b_r]_{\Gamma(S)}.
    \end{equation*}
    However, since a vertex of $\phi_r(S)$ whose hyperlink meets $\Delta$ must have hyperlink contained in $\Delta \cup \{b_r^{-1}, b_{r-1}\}$, a vertex of $\phi_r(S)$ is of hyperlink type $(\Delta, b_r^{-1}; \Sigma \cup \Pi_r^{-1})$ if and only if it is of hyperlink type $(\Delta, \Sigma \cup \Pi_{r-1}; b_{r-1})$.
    We therefore have
    \begin{equation*}
      [\Delta, \Sigma \cup \Pi_r; b_r]_{\Gamma(S)} = [\Delta, \Sigma \cup \Pi_{r-1}; b_{r-1}]_{\Gamma(\phi_r(S))}.
    \end{equation*}

    Given that $\phi_r = (\Delta \cup b_r, b_r)$ reduces $S$, it follows immediately that
    \begin{equation*}
      [\Delta, b_r; \Sigma \cup \Pi_r]_{\Gamma(S)} > [\Delta, \Sigma \cup \Pi_r; b_r]_{\Gamma(S)}.
    \end{equation*}
    Using the above equalities and inequalities, we then have
    \begin{equation*}
      [\Delta, b_{r-1}; \Sigma \cup \Pi_{r-1}]_{\Gamma(\phi_r(S))} > [\Delta, \Sigma \cup \Pi_{r-1}; b_{r-1}]_{\Gamma(\phi_r(S))},
    \end{equation*}
    which is equivalent to saying that $\phi_{r-1} = (\Delta \cup b_{r-1}, b_{r-1})$ reduces $\phi_r(S)$.

    To see that $\phi_{i}$ reduces $\phi_{i+1} \cdots \phi_r (S)$, note that any hyperedge of $\Gamma(\phi_{i+1} \cdots \phi_r (S))$ incident to $\Delta$ is contained in $\Delta \cup \{b_{i+1}^{-1}, b_i\}$.
    A similar argument shows that any vertex whose hyperlink contributes to capacity and not degree in $\phi_{i+1} \cdots \phi_r (S)$ came from a vertex with hyperlink contributing to capacity and not degree in $\phi_{i+2} \cdots \phi_r(S)$, and similarly for vertices contributing to degree and not capacity.
    It then follows that $\phi_i$ reduces $\phi_{i+1} \cdots \phi_r (S)$.

    Since the net effect of $\phi_1 \cdots \phi_r$ is to multiply the edges in $\Delta$ by the entire word $b$, it is immediate that $\phi_1 \cdots \phi_r(S)$ immerses onto $T$.
  \end{proof}

  By Proposition \ref{Solie--Prop--multiply by b}, if $\phi_r = (\Delta \cup b_r, b_r)$ reduces $S$, then we have an entire sequence of reducing Whitehead automorphisms which, when applied to $S$, yield an $X$-digraph that again immerses onto a graph with Property $(S)$.
  Therefore, whenever $S$ admits an immersion onto a graph satisfying Property $(S)$, some element of $\min(S)$ is guaranteed to also admit an immersion onto a graph satisfying Property $(S)$.
\end{proof}

\begin{cor}
  Let $F(X)$ be a free group with $\#X \geq 3$.
  Then $\AO(\mathcal{SV})$ is decidable.
\end{cor}

\begin{proof}
  If $H \fg F(X)$ is such that $\phi(H) \leq K \in \mathcal{SV}$, then $S_X(\phi(H))$ immerses onto an $X$-digraph satisfying Property $(S)$.
  Therefore some element of $\min(S_X(\phi(H)))$ immerses onto an $X$-digraph satisfying Property $(S)$; equivalently, some element of $\min(S_X(\phi(H)))$ has a principal quotient satisfying Property $(S)$.
  Since $\min(S_X(\phi(H))) = \min(S_X(H))$, some element of $\min(S_X(H))$ has a principal quotient satisfying Property $(S)$.

  Our algorithm is therefore:
  \begin{alg}
    Given $H \fg F(X)$, we may determine whether or not there exist $\phi \in \Aut F(X)$ and $K \in \mathcal{SV}$ such that $\phi(H) \leq K$ as follows:
    \begin{enumerate}
      \item
        Construct the finite $X$-digraph $S_X(H)$.
      \item
        Construct the finite set $\min(S_X(H))$.
      \item
        Construct the finite set $\PQ(\min(S_X(H))) := \ds{\bigcup_{M \in \min(S_X(H))} \PQ(M)}$.
      \item
        For each $P \in \PQ(\min(S_X(H)))$, determine whether or not $P$ satisfies Property $(S)$.
        If a $P$ satisfying Property $(S)$ is found, conclude that there exist $\phi \in \Aut F(X)$ and $K \in \mathcal{SV}$ such that $\phi(H) \leq K$.
        Otherwise, conclude that no such $\phi \in \Aut F(X)$ and $K \in \mathcal{SV}$ exist.
    \end{enumerate}
  \end{alg}
\end{proof}

Since $\mathcal{SV}$ is the set of subgroups of $F(X)$ which are vertex groups in some very small elementary cyclic splitting of $F(X)$, we have the following theorem.

\begin{thm}
  \label{Solie--Theorem--Filling decidable for segment splittings}
  Let $F(X)$ be a free group of finite rank at least three.
  There is an algorithm to determine, given $H \fg F(X)$, whether or not $H$ is elliptic in a nontrivial, very small elementary cyclic splitting of $F(X)$.
\end{thm}

\subsubsection{The Rank Two Case}

Let $F(a,b)$ denote the free group of rank two.
The following characterization of the standard vertex subgroups of $F(a,b)$ follows directly from Proposition \ref{Solie--Prop--Characterization of vertex subgroups}.

\begin{prop}
  For $F(a,b)$, we have $\mathcal{SV} = \emptyset$ and $\mathcal{LV}=\{\langle a, a^b \rangle, \langle b, b^a \rangle\}$.
\end{prop}

As both elements of $\mathcal{LV}$ share an automorphic orbit, the problem of determining ellipticity in a vertex subgroup for $F(a,b)$ is therefore simply the problem $\AO(\langle a, a^b \rangle)$.

\begin{thm}
  \label{Solie--Theorem--ASP in F2 case}
  The problem $\AO(\langle a, a^b \rangle)$ is decidable.
\end{thm}

\begin{proof}
  Suppose that $H \fg \langle a, a^b \rangle$ is cyclically reduced and that $H$ is contained in no proper free factor of $F(a,b)$.
  We have an immersion $S_X(H) \rightarrow S_X(\langle a, a^b \rangle)$.
  Note that if this immersion is not a surjection, then $H$ is contained in a proper free factor of $F(a,b)$ (either $\langle a \rangle$ or $\langle a^b \rangle$).

  Every vertex of $S=S_X(H)$ therefore has a hyperlink which is a subset of either $\{a, a^{-1}, b\}$ or $\{a, a^{-1}, b^{-1}\}$.
  In particular, note that the set of initial vertices of $b$-edges in $S_X(H)$ is disjoint from the set of terminal vertices of $b$-edges, and so $S_X(H)$ has at least $2\#E_b S$ vertices.

  \begin{figure}
    \begin{center}
      \subfigure[$S_X(\langle a, a^b \rangle)$]{
        \includegraphics{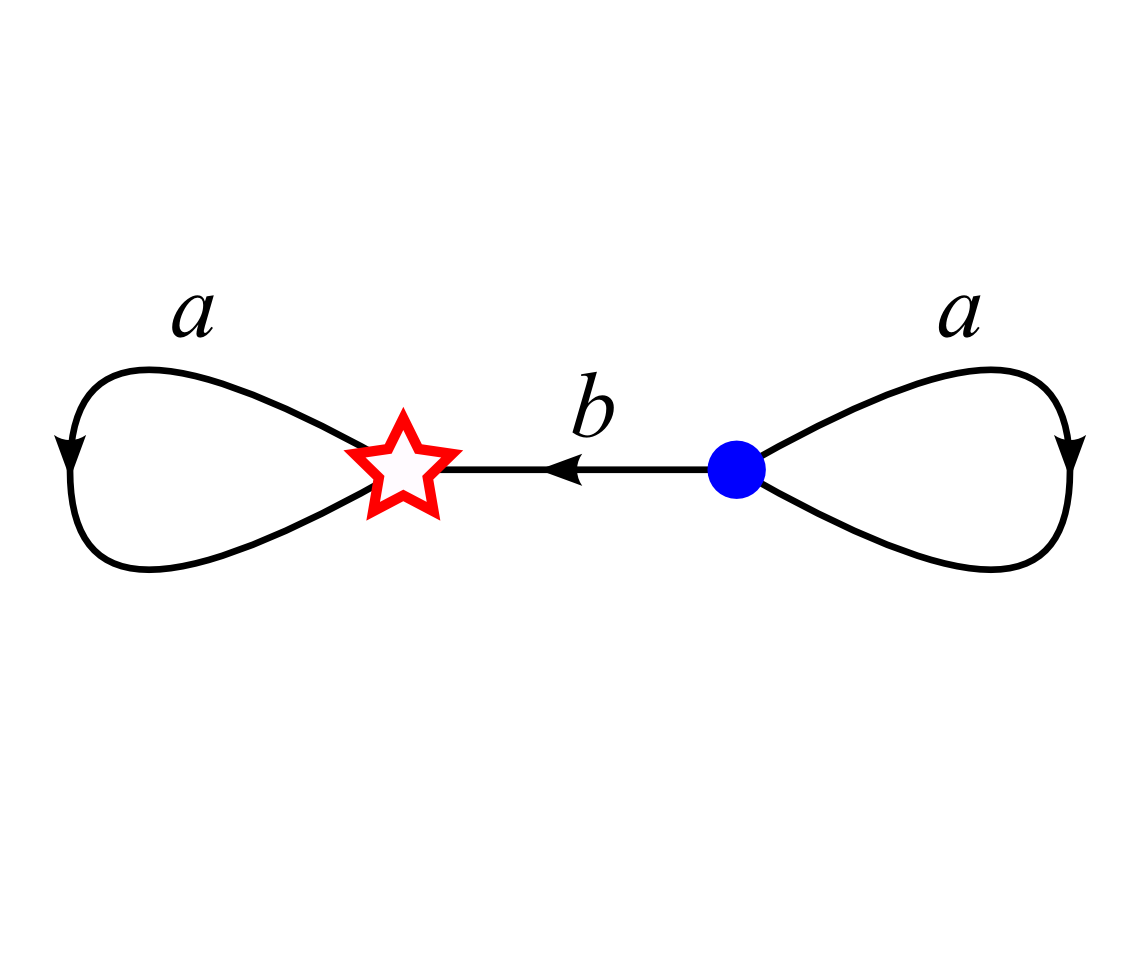} }
      \subfigure[$\Gamma(S_X(\langle a, a^b \rangle))$]{
        \includegraphics{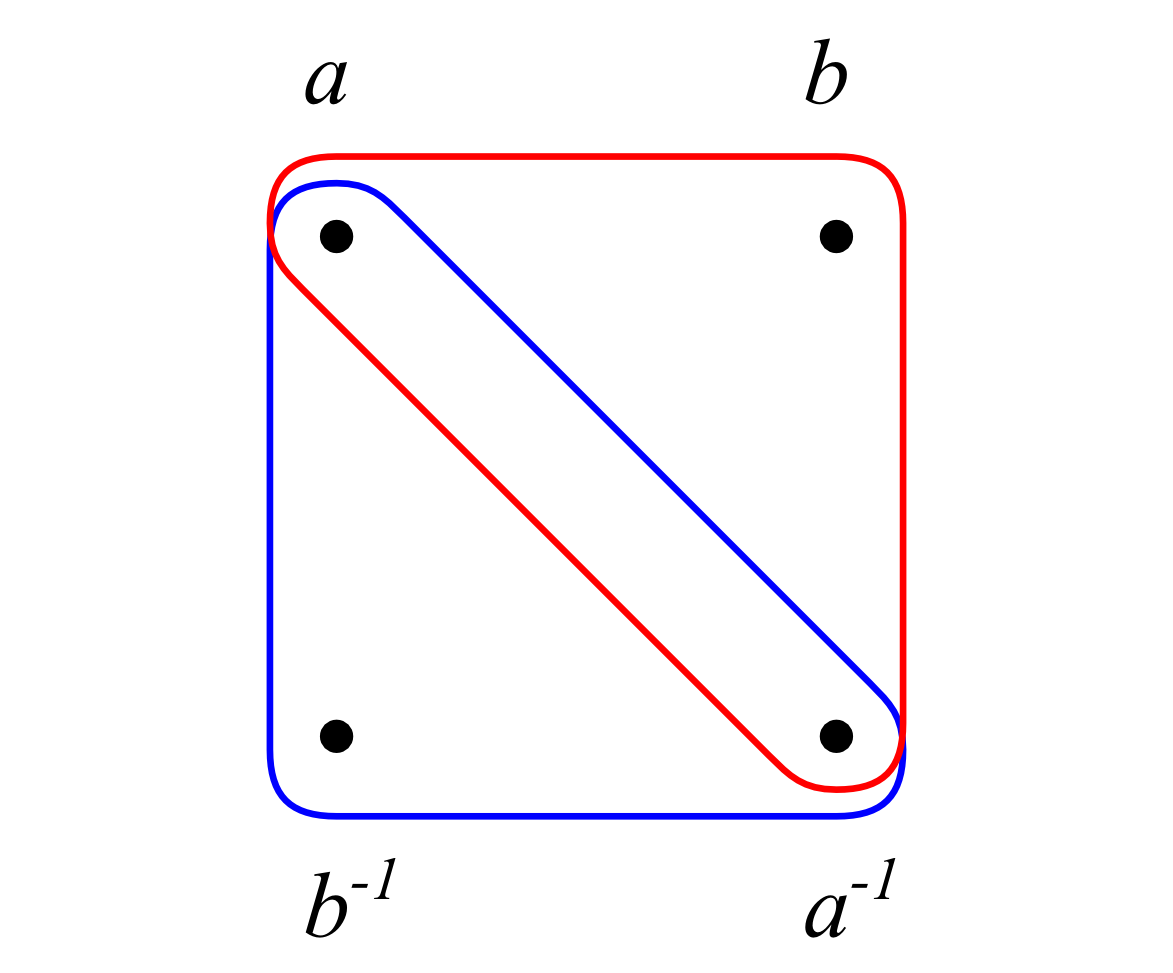} }
      \subfigure[$\Gamma(S_X(H))$ for $H \fg F(X)$, showing every possible edge.]{
        \includegraphics{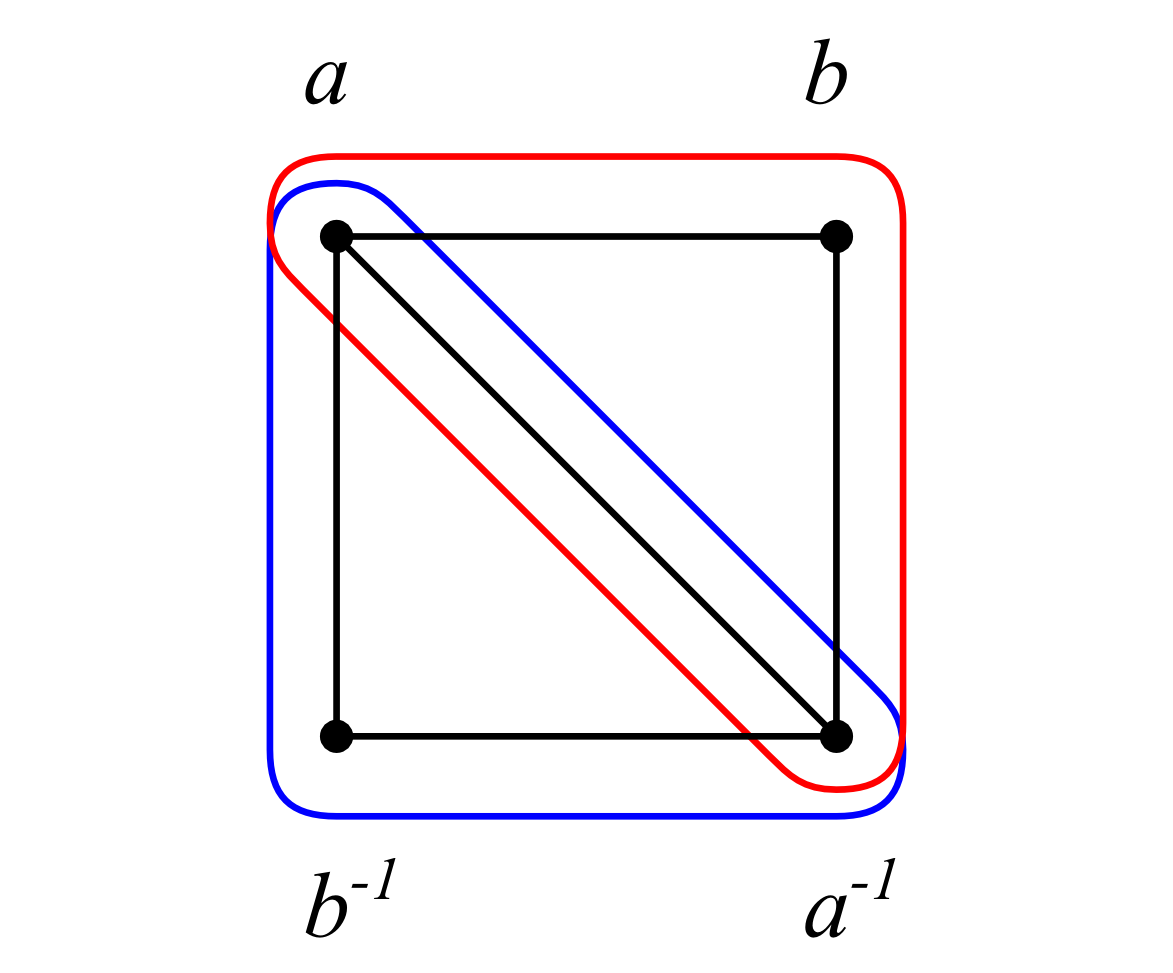} }
    \end{center}
    \caption{Graphs associated to the subgroup $\langle a, a^b \rangle$.}
  \end{figure}

  Suppose that $S - E_b S$ has $k$ connected components.
  Since $S$ is cyclically reduced, each of these components has at least one $a$-edge.
  Since $S-E_bS$ is an $a$-digraph, each connected component of $S-E_b S$ is either a path or a cycle.
  We therefore have $\#E_aS \geq 2\#E_b S - k$, hence $\#E_aS + k \geq 2 \#E_b S$.
  Since each of the connected components of $S-E_bS$ must have at least one $a$-edge, $\#E_aS \geq k$ and therefore $\#E_aS \geq \#E_bS$.
  In terms of the Whitehead hypergraph $\Gamma(S)$, we have $\deg_\Gamma(a) \geq \deg_\Gamma(b)$.

  Now suppose that $\phi = (C,m)$ is a reducing Whitehead automorphism for $H$.
  Since $C$ is an $m$-cut, if $C$ has one or three elements, $\phi(S)=S$.
  $C$ therefore has two elements; without loss of generality, we may assume that $C = \{a,b\}$.

  Suppose that $m=a$, so that $\phi = (\{a,b\},a)$ is reducing.
  Clearly $\phi$ leaves $S_X(\langle a, a^b \rangle)$ invariant, so $\phi(S)$ admits an immersion onto $S_X(\langle a, a^b \rangle)$.

  Suppose that $m=b$, so that $\phi = (\{a,b\},b)$ is reducing.
  Since $\deg_\Gamma(a) \geq \deg_\Gamma(b)$, the Whitehead automorphism $\phi' = (\{a,b\},a)$ is also reducing for $S$.
  By the above observation, $\phi(S)$ also admits an immersion onto $S_X(\langle a, a^b \rangle)$.

  Therefore, if $S$ admits an immersion onto $S_X(\langle a, a^b \rangle)$, there is at least one element of $\min(S)$ which also admits such an immersion.
  It follows directly that an arbitrary subgroup $H \fg F(a,b)$ has some automorphic image which is a subgroup of $\langle a, a^b \rangle$ if and only if some element of $\min(S_X(H))$ admits an immersion onto $S_X(\langle a, a^b \rangle)$.

  We may check whether a given graph $S$ admits an immersion onto $S_X(\langle a, a^b \rangle)$ by calculating its finitely many quotient graphs with only two vertices and determining whether any is isomorphic to $S_X(\langle a, a^b \rangle)$.

  Our algorithm is therefore the following:

  \begin{alg}
    Given $H \fg F(a,b)$, we may determine whether or not there exists $\phi \in \Aut F(a,b)$ such that $\phi(H) \leq \langle a, a^b \rangle$ as follows:
    \begin{enumerate}
      \item
        Construct the finite graph $S_X(H)$;
      \item
        Construct the finite set $\min(S_X(H))$;
      \item
        If some member of $\min(S_X(H))$ admits an immersion onto $S_X(\langle a, a^b \rangle)$, conclude that there is $\phi \in \Aut F(a,b)$ such that  $\phi(H) \leq \langle a, a^b \rangle$.
        Otherwise, conclude that no such $\phi \in \Aut F(X)$ exists.
    \end{enumerate}
  \end{alg}
\end{proof}

\subsubsection{Loop vertex subgroups in higher rank}

Let $F(X)$ be a free group with rank at least three.

Recall that the set $\mathcal{LV}$ is the set of standard loop vertex subgroups; in other words, groups of the form
\begin{equation*}
  \label{Eqn--Standard loop vertex subgroup}
  \langle U, u^v \rangle
\end{equation*}
where $U \sqcup \{v\} = X$ and $u \in \langle U \rangle$ is not a proper power.

Observe that $S_X(\langle U, u^v \rangle)$ has a unique $v$-edge, and that the complement of this edge has at least one component of rank one.

\begin{defn}[Property $(L)$] \index{Property $(L)$}
  \label{Defn--Property (L)}
  Let $S$ be a Stallings graph. We say that $S$ satisfies \emph{Property $(L)$} if
  \begin{enumerate}
    \item
      There is some $x \in X$ for which $S_X(H)$ has a unique $x$-edge $e$; and
    \item
      $S_X(H)-e$ has two connected components, at least of which is a rank one graph.
  \end{enumerate}
  We say that $H \fg F(X)$ satisfies Property $(L)$ if $S_X(H)$ does.
\end{defn}

We note that $S_X(\langle U, u^v \rangle)$ satisfies Property $(L)$.
Suppose $T$ is a cyclically reduced subgraph of $S_X(\langle U, u^v \rangle)$.
It is straightforward to verify that either $T$ satisfies Property $(L)$ or omits some $x \in X$ as an edge label.

\begin{figure}
  \begin{center}
    \includegraphics{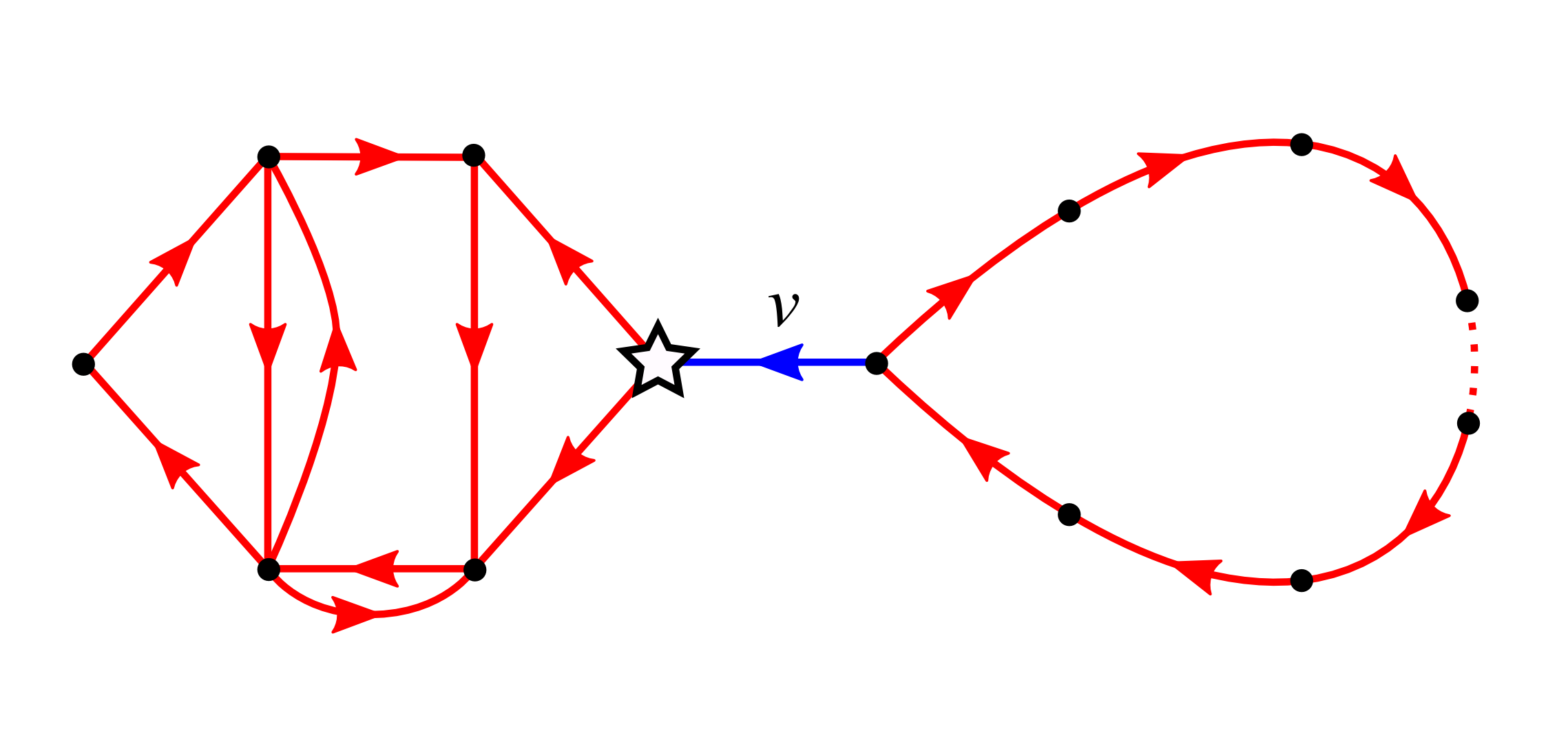}
    \caption[A Stallings graph satisfying Property $(L)$.]{A Stallings graph satisfying Property $(L)$. The red edges have labels different from $v$.}
  \end{center}
\end{figure}

\begin{lem}
  \label{Solie--Lemma--Prop L is WHR invariant}
  Let $S$ be a cyclically reduced $X$-digraph satisfying property $(L)$, and let $\phi=(C,m)$ be a reducing Whitehead automorphism for $S$.
  Then either $\phi(S)$ satisfies Property $(L)$ or $\phi(S)$ omits some element of $X$ as an edge label.
\end{lem}

\begin{proof}
  Let $\phi=(C,m)$ reduce $S$.
  Let $v \in X$ be such that $S$ has a unique $v$-edge $e$ and that $S-e$ has two components, at least one of which is rank one.

  First, suppose that $m \neq v^\pm$.
  Then $\phi(S)$ also has a unique $v$-edge, since $\phi$ changes only the number of $m$-edges.

  Recall that $\phi(S)$ is constructed from $S$ in three stages: subdivision, folding, and leaf deletion.
  Let $S_1$ and $S_2$ be the connected components of $S-e$.
  Let $u_i$ be the endpoint of $e$ in $S_i$ for $i=1,2$.

  We may construct $\phi(S)$ by first subdividing and folding each $S_i$ to obtain a graph $S_i'$.
  It is clear that, since $\phi$ is an automorphism, $S_i$ and $S_i'$ have the same rank.
  We then connect the vertices $u_1$ and $u_2$ via an appropriately oriented path labeled with $\phi(v)$, to obtain a graph $T$.
  By construction, $T$ has at most two unfolded vertices, $u_1$ and $u_2$, and has a unique $v$-edge which is also a cut edge.
  Making the final two folds at $u_1$ and $u_2$ and deleting any leaves which $T$ may have introduces no new paths between the endpoints of the unique $v$-edge, and so $T$ satisfies Property $(L)$.

  Now suppose that $m = v^\pm$.
  Since $\phi$ reduces $S$, it must reduce the number of $v$-edges in $S$.
  Since $S$ has exactly one $v$-edge, $\phi(S)$ must have no $v$-edges, so $\phi(S)$ omits some element of $X$ as an edge label.
\end{proof}

\begin{thm}
  \label{Solie--Theorem--Property L decidable up to automorphism}
  There is an algorithm to determine, given $H \fg F(X)$, whether or not there exists $\phi \in \Aut F(X)$ such that $S_X(\phi(H))$ satisfies Property $(L)$.
\end{thm}

\begin{proof}
  If there exists such a $\phi \in \Aut F(X)$ such that $S_X(\phi(H))$ satisfies Property $(L)$, then by Lemma \ref{Solie--Lemma--Prop L is WHR invariant}, $\min(S_X(\phi(H))=\min(S_X(H))$ has an element which satisfies Property $(L)$.
  Since elements of $\min(S_X(H))$ represent subgroups which are automorphic images of $H$, the converse also holds.
  Therefore, the algorithm is as follows:

  \begin{alg}
    Given $H \fg F(X)$, we may determine whether or not there exists $\phi \in \Aut F(X)$ such that $S_X(\phi(H))$ satisfies Property $(L)$ as follows:
    \begin{enumerate}
      \item
        Construct the finite graph $S_X(H)$.
      \item
        Construct the finite set $\min(S_X(H))$.
      \item
        If some element of the finite set $\min(S_X(H))$ satisfies Property $(L)$, conclude that there exists $\phi \in \Aut F(X)$ such that $S_X(\phi(H))$ satisfies Property $(L)$.
        Otherwise, conclude that no such $\phi$ exists.
    \end{enumerate}
  \end{alg}
\end{proof} 
\section{Free group actions on trees}

\subsection{Outer space}

We now apply our results to the study of free group actions on $\mathbb{R}$-trees.

\begin{defn}[$\mathbb{R}$-tree] \index{$\mathbb{R}$-tree}
  An \emph{$\mathbb{R}$-tree} is a geodesic metric space in which every two points are connected by a unique injective path and this path is a geodesic.
\end{defn}

Recall that the action of $F(X)$ on an $\mathbb{R}$-tree $T$ is:
\begin{itemize}
  \item
    \emph{isometric} if each element $w \in F(X)$ acts as an isometry on $T$;
  \item
    \emph{minimal} if there exists no $F(X)$-invariant subtree of $T$;
  \item
    \emph{very small} if the stabilizer of any tripod is trivial and the stabilizer of any arc is either trivial or maximal cyclic in the stabilizers of the endpoints of the arc;
  \item
    \emph{simplicial} if $T$ has the topological structure of a simplicial complex.
\end{itemize}
We will assume that all actions of $F(X)$ on $\mathbb{R}$-trees are isometric and minimal.

\begin{defn}[Filling subgroup] \index{filling subgroup} \index{filling element}
  \label{Defn--Filling subgroup}
  Let $H \fg F(X)$.
  We say that $H$ is a \emph{filling subgroup} if the set $H$ fixes no point in any very small action of $F(X)$ on an $\mathbb{R}$-tree, and that $H$ is a \emph{non-filling subgroup} if $H$, as a set, fixes a point in some very small action of $F(X)$ on an $\mathbb{R}$-tree.
  We say that $w \in F(X)$ is a \emph{filling element} if $w$ generates a filling subgroup of $F(X)$.
\end{defn}

The work of Guirardel allows one to approximate the very small action of $F(X)$ on a given $\mathbb{R}$-tree by a very small action on a simplicial tree.
In particular, if $H$ fixes a point in the $\mathbb{R}$-tree, then we may have $H$ fix a point in the simplicial approximation \cite[Theorem 1]{Guirardel1998}.

\begin{prop}
  A subgroup $H \fg F(X)$ is non-filling if and only if $H$ fixes a point in some very small action of $F(X)$ on a simplicial tree $T$.
\end{prop}

A very small action of $F(X)$ on a simplicial tree gives a decomposition of $F(X)$ as a graph of groups, the details of which can be found in \cite{Serre2003}.
If a subgroup $H \fg F(X)$ fixes a point in a very small action of $F(X)$ on a simplicial tree, then $H$ is necessarily elliptic in some elementary cyclic splitting of $F(X)$.
The author previously showed that filling elements are generic in $F(X)$ in the sense that a sufficiently long element is overwhelmingly likely to be filling \cite{Solie2010}.

As a direct consequence of the previous section, we find we have algorithms to identify elements which fix no point in certain types of free group actions on simplicial trees.

\begin{cor}
There exists an algorithm to determine whether a given element $w \in F(a,b)$ fixes a point in some very small action of $F(a,b)$ on a simplicial tree.
\end{cor}

\begin{cor}
There exists an algorithm to determine whether a given element $w \in F(X)$ fixes a point in some very small action of $F(X)$ on a simplicial tree whose fundamental domain is a single edge with distinct endpoints.
\end{cor}

\bibliographystyle{plain}
\bibliography{segment}

\end{document}